\documentclass[12pt,oneside, reqno]{amsart}
\usepackage[utf8]{inputenc}
\usepackage{amsmath,amsthm,amssymb}
\usepackage{mathrsfs} 
\usepackage{graphicx, xcolor}
\usepackage{caption} 
\usepackage{subfiles}
\usepackage{newverbs} 
\usepackage{array} 
\usepackage{tikz}
\usepackage[justification=centering]{caption}
\usepackage{float}
\usepackage{braket}
\usepackage{enumerate}
\usepackage{mathtools}
\usepackage{dsfont}
\usepackage{color,soul}
\sethlcolor{yellow}
\usepackage{imakeidx}
\usepackage[utf8]{inputenc}
\usepackage[T1]{fontenc}
\usepackage[colorlinks,citecolor=blue,linkcolor=red!70!black]{hyperref}
\usepackage[nameinlink]{cleveref}
\usetikzlibrary{decorations.pathreplacing, calc}
\usepackage{comment}

\usepackage[marginpar=4cm,left=5cm,right=5cm]{geometry}
\numberwithin{equation}{section}
\usepackage{amssymb} 
\usepackage{graphicx} 

\newverbcommand{\CMRverb}{\tiny\color{blue}}{}
\newverbcommand{\GRverb}{\tiny\color{teal}}{}
\newverbcommand{\STverb}{\tiny\color{cyan}}{}
\newverbcommand{\Pverb}{\tiny\color{violet}}{}
\newverbcommand{\Averb}{\tiny\color{brown}}{}

\theoremstyle{plain}
\newtheorem{THM}{Theorem}[section]

\newtheorem{PROP}[THM]{Proposition}
\newtheorem{LEM}[THM]{Lemma}
\newtheorem{COR}[THM]{Corollary}

\theoremstyle{definition}
\newtheorem{DEF}[THM]{Definition}
\newtheorem{RMK}[THM]{Remark}

\newcommand{\e}{\epsilon}

\renewcommand{\bar}{\overline}
\renewcommand{\|}{|\!|} 
\renewcommand{\r}{q_{n_k+1}}
\newcommand{\x}{(x,s_1)}
\newcommand{\y}{(y,s_2)}
\newcommand{\h}{\mathfrak{h}}
\newcommand{\g}{\mathfrak{g}}
\newcommand{\s}{\tilde{s}}
\newcommand{\N}{\tilde{N}}
\allowdisplaybreaks[4]
\title{Special flow systems with the minimal
 self-joining property}

\author[Zhai]{Yibo Zhai}
\address{DEPARTMENT OF MATHEMATICS, UNIVERSITY OF UTAH}
\email{zhai@math.utah.edu}

\date{\today}
\makeindex
\begin{document}

\begin{abstract}
Arnol’d flows are a class of area-preserving flows on surfaces. In this
paper, we prove that typical Arnol’d flows
have the minimal self-joining property. Consequently, we can classify centralizers
and factors of typical Arnol’d flows.
\end{abstract}
\maketitle
\section{\textbf{Introduction}}
The study of joinings is an important branch of ergodic theory with many interesting connections with other fields.  In \cite{Furstenberg1967}, Furstenberg introduced the concepts of joinings, where he observed deep analogies between the arithmetic of integers and the classification of measure-preserving dynamical systems: we can consider products and factors of systems. 
\begin{DEF}
    A \textit{joining} of two measure-preserving systems \((X, T,\mu)\) and \((Y, S,\nu)\) means a probability measure \(m\) on \(X\times Y\) such that \(m\) is \(T\times S\) invariant and the projections of \(m\) onto \(X\) and \(Y\) are \(\mu\) and \(\nu\) respectively.
\end{DEF}

\noindent It turns out that we can use the classification of joinings to study a factor of the system. This idea is the starting point for studying joinings. Furstenberg showed in \cite{Furstenberg1967} that if two systems \((X, T,\mu)\) and \((Y, S,\nu)\) are disjoint, which means the only joining is \(\mu\otimes \nu, \)  then they don't share common non-trivial factors. In the 1980s, Ratner applied the study of joinings to classify factors of horocycle flows on the unit tangent bundle of a surface of constant negative curvature with finite volume in her celebrated seminal work \cite{Ratner}. In this paper, Ratner introduced an important property called \textit{Ratner's property}, which is central to studying parabolic flows. The Ratner property shows a deep connection between joinings and orbits of nearby points: if the divergence of the distance between nearby points grows but grows slowly under the dynamics of flows, then an ergodic joining of the flows is either finite-to-one on \(a.e.\) fiber or the product measure. The existence of a Ratner-type property in non-homogeneous systems was first observed by Fr\k{a}czek and Lema$\acute{\text{n}}$czyk in \cite{Mildmixing}. Interestingly, some celebrated conjectures can be verified, such as Rokhlin's multiple-mixing conjecture and Sarnak's conjecture, by studying joinings. In \cite{Multiple}, Fayad and Kanigowski proved that a class of surface flows is mixing of all orders by proving \(\textit{Switchable weak Ratner's property}.\) Later, Kanigowski, Lema$\acute{\text{n}}$czyk, and Ulcigrai showed the M$\ddot{\text{o}}$bius orthogonality of time-changes of horocycle flows and Arnol'd flows by proving the disjointness criteria in \cite{Disjoint}. 
It is natural to ask the following question: if two systems only have trivial common factors, are those two systems disjoint? The answer is negative, and a counterexample was constructed by Rudolph in \cite{Rudolph}. In this paper, Rudolph introduced another important property, called the \(\textit{minimal self-joining property}\). For simplicity, we define the \textit{2-fold minimal self-joining property}  here. 
\begin{DEF}
\label{MSJ for action}
We say that a measure-preserving flow system \((X,(T_t)_{t\in \mathbb{R}},\mu)\) has the \textit{2-fold minimal self-joining property }if the only ergodic self-joinings of \((X,(T_t)_{t\in \mathbb{R}},\mu)\) are the product measure and the image of measure \(\mu\) via map \(x\mapsto (x, T_{t_0} x)\) where \(t_0\in\mathbb{R}.\) 
\end{DEF}
For a weakly mixing flow system, the 2-fold minimal self-joining property is equivalent to the minimal self-joining property. It is still unknown whether 2-fold minimal self-joining implies minimal self-joining for measure-preserving transformations. For definitions of the \(k\)-fold minimal self-joining property and the minimal self-joining property, see \Cref{Self-joining}. The minimal self-joining property is a rare property in ergodic theory. The existence of the minimal self-joining property is only known for a few systems. In \cite{DelJunco1980}, Junco, Rahe, and Swanson proved Chacon’s automorphism has the minimal self-joining property.
In \cite{Kingrank1}, King proved all mixing rank one systems have the minimal self-joining property. In \cite{IETMSJ}, Ferenczi, Holton, and Zamboni proved the minimal self-joining property for linearly recurrent three-interval exchange transformations. In \cite{MSJFL}, Fr\k{a}czek and Lema$\acute{\text{n}}$czyk proved the minimal self-joining property for a class of special flows when the rotation angle has bounded quotients. In the context of infinite ergodic theory, Klimov and Ryzhikov proved the minimal self-joining property for mixing rank one systems in \cite{MSJR1}. Regarding the classification of factor systems, the minimal self-joining property of a system implies that the system is \textit{prime}. A system is called prime
if it admits no proper factor systems. Due to the rareness of the minimal self-joining property, Veech introduced a weaker but important property called \(\textit{2-simple}\) in \cite{Veech} and gave a criterion for a system to be prime when the system is 2-simple. Typical systems are not 2-simple, which is proved in \cite{AGEEV_2003}. In particular, typical systems don't have the minimal self-joining property. Many systems are prime but don't have the minimal self-joining property, such as some horocycle flows in \cite{Ratner} and rank one systems in \cite{Primesystem}. The absence of the 2-simple property of three interval-exchange transformations is proved in \cite{Self3IET}. Recently, a mild mixing rank one system without the minimal self-joining property was constructed in \cite{rankmildmixingminimal}. The survey \cite{ deLaRue2020} is a good introduction to the study of joining.  For other examples related to minimal self-joining, 2-simple, and prime, see \cite{UD1992,Blanchard1995,Blanchard1998,DA2007,SY2014,parreau2015prime,InfiniteMSJ2018}. 

The problem we are concerned about in this paper is the self-joining property of a class of smooth flows on the 2-dimensional torus. These flows can be represented as special flows built over rotations.
\begin{DEF}
     Let \(R\) be an automorphism of a standard probability space \((X,\mathcal{B},\mu).\) Let \(f\in L_{+}^1(X,\mu)\). The special flow \((T_t^f)_{t\in \mathbb{R}}\) built over \((X,R,\mu)\) and under the ceiling function \(f\) is given by
\begin{equation}
\notag
\begin{aligned}
    X\times\mathbb{R}/\sim &\rightarrow X\times \mathbb{R}/\sim\\
        (x,s)&\rightarrow(x,s+t),
\end{aligned}
\end{equation}
where \(\sim\) is the identification 
\[(x,s+f(x))\sim (R(x),s).\]
\end{DEF}
\noindent The study of special flows dates back to von Neumann's work in \cite{Neumann1932}, where he proved the weakly mixing property when \(R\) is the circle rotation and \(f\) is piecewise absolutely continuous with the sum of jumps not equal to 0. Kochergin proved that this class of flows is never mixing in \cite{Nonmixing}. Later, Arnol'd constructed a class of special flows with asymmetric logarithm singularities in \cite{Arnold}, and Sinai and Khanin proved the Arnol'd flow is mixing in \cite{Sinai}. For the definition of Arnol'd flows, see \Cref{Special flows with singularities}. The main result of this paper is to prove the following theorem:
\begin{THM}
\label{Main}
For an explicit full Lebesgue measure set of \(\alpha\in(0,1)\), Arnol'd flows have the minimal self-joining property.
\end{THM}
\noindent The conditions on Arnol'd flows will be given in \Cref{Special flows with singularities}. We have the following corollary of \Cref{Main}:
\begin{COR}
Under assumptions in \Cref{Main}, Arnol'd flows are prime.   
\end{COR}
As far as we know, it was previously unknown if any Arnol'd flows have this property. The main idea is to further classify finite extension joinings of a class of Arnol'd flows in \cite{Multiple}. Suppose \(\nu\) is a non-trivial finite extension joining of an Arnol'd considered in our paper, then for \(a.e.\) \(x\), it corresponds to \(n\geq 2\) points \(y_1,...,y_n\) on the fiber by the measurable selection. We move \(x\) along the trajectories of the Arnol'd flow to a neighborhood of \(x\), and obtain a new point \(x'=T_{t_0}^fx.\) The fiber points of \(x'\) are \(y_1',...,y_n'.\) If \(x\) and \(x'\) are in some nice set, then each fiber point \(y_i\) is close to the fiber point \(y_i'.\) By invariance of the joining \(\nu\), there is a permutation \(\sigma\) on set \(\{1,...,n\}\) so that \(T_{t_0}^fy_i=y_{\sigma(i)}'\). We prove that such a permutation cannot exist, thereby ruling out non-trivial finite extension joinings. The distances between \(y_i\) and \(y_i'\) for \(1\leq i\leq n\) will give some constraints to the permutation \(\sigma.\) By shearing of orbits between nearby points \(y_i\) and \(y_i'\), we introduce new variants of Ratner's property arguments which allow us to give a precise distance between \(y_i\) and \(y_i'.\) If that permutation happens, there should be a relative shift distance between the base point \(x\) and its fiber points \(y_i\). Motivated by ideas in \cite{MSJFL}, we prove a result that allows us to compare orbits of pairs of points with arbitrary distances in \Cref{Large Shearing}, which gives an upper bound for the relative shift distance. However, this upper bound will result in contradictions with the constraints on the permutation \(\sigma\). The classification of centralizers uses similar ideas, and it is easier. The difficult part of this paper is that we need to control the orbits of points at arbitrary distances rather than only at very close distances. Moreover, we need to simultaneously handle different types of shearing and control the orbits of multiple pairs of points.
\\\\
\noindent \textbf{Relevant Questions:} It might be interesting to prove the minimal self-joining property for other surface flows or special flows. For special flows with asymmetric power like singularities, such as Kochergin flows, some ideas in this paper may still apply due to the strong shearing phenomenon. However, for symmetric singularities or higher genus cases, there may be some ``cancellation'' to produce non-trivial joinings. For von Neumann flows, it is also unclear whether they are 2-simple in typical cases due to the weaker shearing effect.
\\\\
\noindent \textbf{Outline of Paper:} In \Cref{Self-joining}, we present the definition of the minimal self-joining, and the switchable weak Ratner's properties. Our proof will use the fact that typical Arnol'd flows have the \textit{finite extension joining property}, which will be deduced from the switchable weak Ratner's property.

In \Cref{Special flows with singularities}, we give the conditions of the ceiling function and rotation angle \(\alpha.\) Those conditions will imply the finite extension joining and the switchable weak Ratner's properties in \Cref{Self-joining}.

In \Cref{Centralizer}, we prove that the centralizers of the special flows under the conditions in \Cref{Special flows with singularities} are trivial. Based on the results in \cite{Multiple} and \cite{Disjoint}, we develop some tools to compare the different orbits of a typical pair of points when their distance is arbitrary. By ergodicity of the flow, if a pair of points lies in some Lusin set, then we can rule out non-trivial centralizers.

In \Cref{General case of self-joining}, we show that all the finite-to-one joinings are actually one-to-one under conditions in \Cref{Special flows with singularities}. First, by shearing of orbits between nearby points, we develop variants of Ratner's properties in \Cref{Product Criteria} and \Cref{Simpleness Criteria}. If the joining is not one-to-one, then we have a certain permutation among those finitely many points on the fiber due to Ratner's properties. However, by iterating the above arguments finitely many times and use the estimates in \Cref{Large Shearing}, we can obtain a contradiction to the existence of the permutation. Together with the result in \Cref{Centralizer}, the minimal self-joining property follows.
\\\\
\noindent \textbf{Acknowledgments:} The author thanks Professor Jon Chaika for introducing this problem and for his constant help, patience, and instructions. The author thanks Krzysztof Fr\k{a}czek, Adam Kanigowski, Mariusz Lema$\acute{\text{n}}$czyk, and Corinna Ulcigrai for valuable discussions. The author also thanks the referees for their suggestions and comments to improve the readability of this paper. 
\section{\textbf{Self-Joining}}
\label{Self-joining}
In this section, we give the general definition of the minimal self-joining property and list some self-joining properties we will use later.
 Let \((T_t)_{t\in \mathbb{R}}\) be an ergodic flow acting on a standard probability Borel space \((X,\mathcal{B},\mu)\). Define \(\mathcal{T}=(X,(T_t)_{t\in \mathbb{R}},\mu)\).
\begin{DEF}
A $k$-self-joining of $\mathcal{T}$ is a probability measure $m$ on 
\(\prod_{i=1}^k X_i\) where \((X_i,(T_t)_{t\in \mathbb{R}},\mu)=\mathcal{T}\),
which is $(T_t \times \ldots \times T_t)_{t\in \mathbb{R}}$-invariant and $({\pi_i}) _* m = \mu$ where $\pi_i$ is the natural projection onto the $i$-th coordinate.
\end{DEF}
Denote \(C(\mathcal{T})\) to be the collection of automorphisms \(R: (X,\mu)\rightarrow (X,\mu)\) such that \(T_t\circ R =R\circ T_t\) for every \(t\in\mathbb{R}\). 
\begin{DEF}
We say that a \(k\)-self-joining \(m\) of \(\mathcal{T}\) is \textit{off-diagonal} if there are \(\phi_i\in C(\mathcal{T})\), \(1\leq i \leq k\), so that \(m\) is the image of \(\mu\) via map:
\[x\mapsto (\phi_1(x),\phi_2(x),...,\phi_k(x)).\]
\end{DEF}
\begin{DEF}
We say that \(\mathcal{T}\) is  \(\textit{k-simple}\) if every ergodic \(k\)-self-joining \(m\) is a product of off-diagonal measures. That is, there exists a partition \(P=\{P_1,...,P_{j}\}\) of \(\{1,2,...,k\}\) such that for each \(1\leq l \leq j\), \(({\pi_{P_{l}}})_{*}m\) is an off-diagonal measure and \(m=\prod_{l=1}^{j}({\pi_{P_{l}}})_{*}m\) where \(\pi_{P_{l}}\) is the projection on \(\prod_{i\in P_{l}} X_i.\)
\end{DEF}
\begin{DEF}
We say that \(\mathcal{T}\) has the \textit{k-fold minimal self-joining property} if it's \(k\)-simple and \(C(\mathcal{T})=\{T_t:t\in \mathbb{R}\}.\) If \(\mathcal{T}\) has \(k\)-fold minimal self-joining property for every \(k\geq2\), then we say \(\mathcal{T}\) has the \textit{minimal self-joining property.}
\end{DEF}
When the system \(\mathcal{T}\) is mixing, to show \(\mathcal{T}\) has the minimal self-joining property, it suffices to check the 2-fold minimal self-joining property.
\begin{PROP}\cite{2Simple}
If \(\mathcal{T}\) is a \(\textit{mixing}\) flow system, then the 2-fold minimal self-joining property implies the minimal self-joining property. 
\end{PROP}
In the following, when we say a self-joining of a system \(\mathcal{T}\), we mean a 2-self-joining of \(\mathcal{T}.\) If a self-joining \(m\) is an off-diagonal measure, we will also say \(m\) is 2-simple.
\begin{DEF}
We say an ergodic self-joining \(m\) is a \textit{finite extension joining} of $\mathcal{T}$ if the natural projection $\pi$ is finite-to-one almost everywhere. If the only ergodic self-joinings of \(\mathcal{T}\) are the product measure and finite extension joinings, we say \(\mathcal{T}\) has the \textit{finite extension joining property.}
\end{DEF}
One can combine the finite extension joining property with the mixing property of a system \(\mathcal{T}\) to conclude that the system is multiple mixing. To show the finite extension property, it suffices to prove the SWR property as follows:
\begin{DEF}
\label{SWR}
Let \(K\subset \mathbb{R}\setminus\{0\}\) be a given compact set. Fix \(t_0\in \mathbb{R}_{+}.\) We say \(\mathcal{T}\) has the \(\textit{switchable}\) \(R(t_0,K)\) \(\textit{property}\) if for every \(\e>0\) and \(N\in \mathbb{N}\) there exist \(\kappa=\kappa(\e),\) \(\delta=\delta(\e,N)\) and a set \(Z=Z(\e,N)\in\mathcal{B}\) with \(\mu(Z)>1-\e\) such that for any \(x,y\in Z\) with \(d(x,y)<\delta\), \(x\) not in the orbit of \(y\) there exist \(M=M(x,y),\) \(L=L(x,y)\in \mathbb{N}\) with \(M,L>N\) and \(\frac{L}{M}\geq \kappa\) and \(p=p(x,y)\in K\) such that
\begin{equation}
\notag
    \frac{1}{L}|n\in[M,M+L]: d(T_{nt_0}(x),T_{nt_0+p}(y))<\e|>1-\e
\end{equation}
or
\begin{equation}
\notag
    \frac{1}{L}|n\in[M,M+L]: d(T_{n(-t_0)}(x),T_{n(-t_0)+p}(y))<\e|>1-\e.
\end{equation}
If the set of \(t_0\) such that \(\mathcal{T}\) has the switchable \(R(t_0,K)\)-property is uncountable, we say \(\mathcal{T}\) has \(\textit{SWR-property}.\)
\begin{RMK}
For initial Ratner's property in \cite{Ratner}, \(K=\{t_0,-t_0\}.\) However, for Ratner type properties developed in \cite{Mildmixing} and \cite{Multiple}, the compact set \(K\) is uniform for \(t_0.\)
\end{RMK}
\begin{THM}\cite[Theorem 4]{Multiple}
Let \(\mathcal{T}\) be a weakly mixing flow system. Suppose \(\mathcal{T}\) has the SWR-property, then \(\mathcal{T}\) has the finite extension joining property.   
\end{THM}
\end{DEF}

\section{\textbf{Special flows with singularities}}
\label{Special flows with singularities}
In this section, we give definitions of Arnol'd flow and its special flow representation. For the readers' convenience, we list the conditions mentioned in \cite{Multiple} and \cite{Disjoint}. The special flow we consider is at the end of this section.
Let \(R\) be an automorphism of a standard probability space \((X,\mathcal{B},\mu).\) Let \(f\in L^1(X,\mu)\). The special flow \((T_t^f)_{t\in \mathbb{R}}\) built over \((X,R,\mu)\) and under the ceiling (roof) function \(f\) is given by
\begin{equation}
\label{Definition of flow}
   T_{t}^f(x,s)=(R^mx,t+s-f^{(m)}(x))
\end{equation}
where \(m\) is the unique integer such that 
\(f^{(m)}(x)\leq t+s< f^{(m+1)}(x)\)
and 
\[
f^{(m)}(x) = \begin{cases}
f(x) + \ldots + f(R^{m-1}x) & \text{if } m > 0 \\
\qquad\qquad0 & \text{if } m = 0 \\
-(f(R^mx) + \ldots + f(R^{-1}x)) & \text{if } m < 0.
\end{cases}
\]
So we get a probability measure preserving flow space \index{\(X^f\): The special flow space}\((X^f,(T_t^f)_{t\in\mathbb{R}},\mu^f)\), where \(X^f=\{(x,s): x\in X, 0\leq s<f(x)\}\), \index{\(\mu^f\): The normalized measure of special flow space} \(\mu^f\) is the normalized measure of \(\mu\times \lambda_{\mathbb{R}}\) restricted to \(X^f\).
\index{\(Z^f:\) Measurable set in the special flow space} For a measurable set \(Z\subset X\), denote \(Z^f\coloneqq \{(x,s): x\in Z, 0\leq s<f(x)\}.\)

\begin{RMK}
Suppose \(R\) preserves a unique probability measure \(\mu\), then the special flow \((T_t^f)_{t\in\mathbb{R}}\) preserves a unique probability measure \(\mu^f.\)
\end{RMK}
Arnol'd flows are local Hamiltonian flows whose special flow representations have asymmetric logarithmic singularities roof function \(f\) and irrational rotation base action \(R.\) In general, the base action \(R\) is an interval exchange transformation, see \cite{IETBASE}, but we only consider the rotation case in our paper.

Let \(X=\mathbb{T}\), \(R=R_{\alpha}\) be an irrational circle rotation and \(\mu\) be the Haar measure on \(\mathbb{T}\). 
For any \(\alpha\in \mathbb{R}\setminus\mathbb{Q}\), we can consider the continued fraction expansion:
\[
\alpha= a_0+\frac{1}{a_1+\frac{1}{a_2+\frac{1}{...}}}.
\]
Let \[\frac{p_n}{q_n}\coloneqq
a_0+\frac{1}{a_1+\frac{1}{a_2+...+\frac{1}{{a_{n-1}+\frac{1}{a_n}}}}}\] be \(n\)-\(th\) term in the  sequence of \(\textit{best rational approximations}\) of \(\alpha.\) We call \(\{q_n\}_{n\in \mathbb{N}}\) \textit{denominators} of \(\alpha.\)
We can compute \(q_n\) inductively using following equation:
\[q_{n+1}=a_{n+1}q_n+q_{n-1}.\]
For \(x\in \mathbb{T}\), denote:
\[\|x\|=\min\{\{x\},\{1-x\}\},\]
where \(\{x\}\) is the fraction part of \(x.\)
From \cite{Khinchin}, we have the following estimate: 
\begin{equation}
\label{special time estimate}
    \frac{1}{q_n+q_{n+1}}<\|q_n\alpha\|<\frac{1}{q_{n+1}}.
\end{equation}
We also have the following two classical results.
\begin{THM}[Denjoy-Koksma]
\label{Classical Denjoy Koksma}  
Let \(\phi: \mathbb{T}\rightarrow \mathbb{R}\) be a function of bounded variation, \(\alpha \in \mathbb{R}\setminus \mathbb{Q}\). Then, 
\[\left|\phi^{(q_n)}(x)-q_n \int_{\mathbb{T}} \phi d m_{\mathbb{T}}\right | \leq Var(\phi)\]
for every \(x\in \mathbb{T}\) and denominators \(q_n\) of \(\alpha\) where \(m_{\mathbb{T}}\) is the Haar measure on \(\mathbb{T}.\)
\end{THM}
\begin{LEM}[Ostrowski expansion]
\label{Classical  Ostrowski expansion}
Let \(\alpha\) be a given irrational number. Let \(a_n\) and \(q_n\) be numbers associated with \(\alpha\) as above. Every positive integer \(m\) can be written uniquely as
\[m= \sum_{n=1}^k b_n q_n,\]
where \(0\leq b_n \leq a_n\) and if \(b_n=a_n\) then \(b_{n-1}=0.\)
\end{LEM}
The conditions on rotation angle \(\alpha\) and roof function \(f\) that guarantee the finite extension joining property are given below. For details, we refer to \cite{Multiple} and \cite{Disjoint}.
The assumptions on \(\alpha\), the singularities and the roof function \(f\) are given by the following:\\
\textbf{Assumptions on the base rotation.} For \( \alpha\in \mathbb{R}\setminus\mathbb{Q}\) \index{\(\alpha\): Rotation angle}, let \(\{q_n\}_{n\in\mathbb{N}}\) \index{\(q_n\): Denominators} be the sequence of the denominators of the best rational approximations of \(\alpha\). 

\begin{DEF}
    We say \(\alpha\in \mathcal{D}'\) if \(\alpha\in \mathbb{T}\) and \(\alpha\) satisfies:\\
\((a)\) Let  \(K_{\alpha}:=\{n\in \mathbb{N} : q_{n+1}<q_n\log^{\frac{7}{8}}(q_n)\}\). Then
\begin{equation}
\notag
    \sum_{i\notin K_{\alpha}}\frac{1}{\log^{\frac{7}{8}}q_i}<\infty.
\end{equation}
\((b)\) There exists \(\tau\geq 0\) such that
\begin{equation}
\notag
    q_{n+1}<C_{\alpha} q_n^{1+\tau}
\end{equation}
where \(C_{\alpha}\) is a  constant depending only on \(\alpha\).
\end{DEF}

\begin{LEM}\cite[Proposition 1.7]{Multiple}
    The set \(\mathcal{D}'\) has the full Lebesgue measure in \(\mathbb{T}.\)
\end{LEM}
\noindent\textbf{Assumptions on singularities.}
\begin{DEF}
\label{Assumption on Singularties}
Given $\alpha \in \mathbb{R} \setminus \mathbb{Q}$, we say that 
\index{\(\Xi\): Set of singularities} $\Xi = \{a_1, \ldots, a_k\}$ is \textit{badly approximable by} $\alpha$ if there exists \index{\(C\): Constant related to badly approximable singularities} $C > 1$ such that for every $x \in \mathbb{T}$ and every $s \in \mathbb{N}$, there exists at most one $i_0 \in \{0,1, \ldots, q_s - 1\}$ such that
\begin{equation}
\notag
x + i_0\alpha \in \bigcup_{i=1}^{k} \left[-\frac{1}{2Cq_s} + a_i, a_i + \frac{1}{2Cq_s}\right].
\end{equation}
\end{DEF}
For the rest of the paper, we will assume \(\Xi=\{0\}.\) Therefore, we can take \(C=2\) in \Cref{Assumption on Singularties}.\\
\noindent\textbf{Assumptions on the roof function.}
\begin{DEF}
\label{Definition of roof}
Let the roof function \(f\) be a positive function of the form
\begin{equation}
\label{Roof function}
    f(x)=-A_{-}\log(x)-A_{+}\log(1-x)+g(x) 
\end{equation}
where \index{$A_{-}$: Constant in the roof function} \index{$A_{+}$: Constant in the roof function}
\index{$a$: Lower bound of the roof function}\(A_{-}>A_{+}\geq 0\) and  \(g(x)\in C^2(\mathbb{T}).\) \\
In particular, when \(x\neq 0,\) \(f(x)\geq a\) for some constant \(a>0\).
For simplicity, we will also assume \(\int_{\mathbb{T}} f d\lambda =1.\)
\end{DEF}

With the above conditions, the Arnol'd flows have SWR property as in \Cref{SWR}.

\begin{THM}\cite[Proposition 3.3]{Multiple}
Suppose \(\alpha \in \mathcal{D}'\) and \(f\) satisfies condition in \Cref{Definition of roof}, then the measure preserving system \((X^f,(T_t^f)_{t\in \mathbb{R}},\mu^f)\) has SWR-property. In particular, the special flow has the finite extension property.
\end{THM}
\begin{RMK}
In \cite{Multiple} and \cite{IETBASE}, the roof function \(f\) is allowed to have multiple singularities. However, we focus only on one singularity case.
\end{RMK}


\section{\textbf{Centralizer of the special flow}}
\label{Centralizer}
\subsection{Classification of centralizer}
First, we show that the centralizer of \((T_t)_{t\in \mathbb{R}}^f\) must be trivial for \(\alpha\) in a full Lebesgue measure set.
In the course of the proof, we need to make some assumptions on \(\alpha\). 
\begin{DEF}
 \index{\(\mathcal{D}\): Diophantine condition } 
\label{Definition of rotation angle}
    We say \(\alpha\in \mathcal{D}\) if \(\alpha\in \mathbb{T}\) and \(\alpha\) satisfies:\\
\((\mathcal{D}_1)\) Let  \(K_{\alpha}:=\{n\in \mathbb{N} : q_{n+1}<q_n\log^{\frac{7}{8}}(q_n)\}\). Then
\begin{equation}
\sum_{i\notin K_{\alpha}}\frac{1}{\log^{\frac{7}{8}}q_i}<\infty.
\end{equation}
\((\mathcal{D}_2)\) For every \(n\in \mathbb{N}\)
\begin{equation}
q_{n+1}<C_{\alpha} q_n \log^{3/2} q_n
\end{equation}
where \(C_{\alpha}\) is a  constant depending only on \(\alpha\).\\
\((\mathcal{D}_3)\) There exists a subsequence \((n_k)_{k\geq 1}\), such that
\begin{equation}
q_{n_k+1}\geq q_{n_k}\log q_{n_k} \log \log q_{n_k}.
\end{equation}
\end{DEF}
\begin{THM}
\label{Theorem of centralizer}
Suppose \(\alpha \in \mathcal{D}\) and \(f\) satisfies the condition in \Cref{Definition of roof}, then the centralizers of \((X^f,(T_t^f)_{t\in \mathbb{R}},\mu^f)\) are \(T_{t_0}^f\) where \(t_0\in \mathbb{R}.\)
\end{THM}
\noindent By Khinchin’s theorem, \(\mathcal{D}_2\) and \(\mathcal{D}_3\) are of the full Lebesgue measure. It follows from \cite[Appendix]{Multiple} that \(\mathcal{D}_1\) also has full Lebesgue measure.
Before we give the proof, let us introduce the notations that will be used later. Let \(\alpha\in \mathcal{D}\) be given, we denote
\index{\(H\): Constant \(H\)}
\[H\coloneqq (A_{-}-A_{+})-\min\{\frac{1}{100}, \frac{A_{-}-A_{+}}{10}\},\]
\index{\(P\): The compact set away from 0}
\begin{equation}
\tag{*}
\label{The compact set}
  P\coloneqq [-2H-2, -\frac{H}{200}] \cup [\frac{H}{200}, 2H+2],
\end{equation}
and
\index{\(d_k\): Distance between orbits of \(x\) and \(y\)}
\[d_k=d_k(x,y)\coloneqq \min\limits_{-q_{n_k}<j<q_{n_k}}d(x,R_{\alpha}^j y).\]
We introduce \(d_k\) to represent the minimal distance between orbits \(x\) and \(y\) up to special times \(q_{n_k}.\) For a pair of points \((x,s_1)\) and \((y,s_2)\) in \(X^f\), the orbits of those two points under the special flow at some special times are determined by the quantity \(d_k\). Due to the singularity of the roof function, even a tiny change of \(d_k\) will cause completely different behaviors of these two orbits. Since we do not have any prior information about \(x\) and \(y\), we have to discuss all possibilities of \(d_k.\)
\begin{DEF}
\label{Definition of good pair}
We say \((x,y)\in \mathbb{T}^2\) is a \(\textit{small pair of order k}\) if \(d_k\in (0, \frac{1}{\r\log \r }].\) 
We say \((x,y)\in \mathbb{T}^2\) is a \(\textit{close pair of order k}\) if \(d_k\in[\frac{5}{6\r}, \frac{1}{q_{n_k} \log q_{n_k}}).\) We say \((x,y)\) is a \(\textit{good pair of order k}\) if it's either a small or close pair of order \(k\).
\end{DEF}

\subsection{Step 1}
Let \(\nu\) be a 2-simple self-joining of the special flow \((X^f,(T_t)_{t\in \mathbb{R}}^f,\mu^f)\) built from a centralizer, which means there is a measurable function \(\phi:X^f \rightarrow X^f \) such that the set \index{\(\Omega\): A full measure set in graph joining case}
\begin{equation}
\label{Carrying set of graph joining}
\Omega\coloneqq \{((x,s), \phi(x,s)): (x,s)\in X^f\}
\end{equation}
has measure \(\nu(\Omega)=1\) and 
\(\phi(T_t^f(x,s))=T_t^f(\phi(x,s))\) for every \(t\in \mathbb{R}\) and \(a.e.\) \((x,s)\in X^f\).
Suppose that the centralizer is not trivial, this means that \(x\) and \(y\) are not on the same orbit where \(y\) is the first coordinate of \(\phi(x,s).\)
The first step is to understand the differences between Birkhoff sums 
\begin{equation}
\label{Birkhoff differences}
 f^{(m)}(x)-f^{(m)}(y).   
\end{equation}
We will deal with two cases separately, we call them \(\textit{small shearing}\) (see \Cref{Small Shearing}) and \(\textit{large shearing}\) (see \Cref{Large Shearing}) according to the distance between orbits of \(x\) and \(y.\) In the small shearing case, we don't know when orbits of \(x\) and \(y\) will diverge, but we know the divergence distance between orbits \(x\) and \(y\) will be in the given compact set \(P\). In the large shearing case, the divergence distance may not be in the given compact set, but we know the definite times when Birkhoff sum divergence occurs.
In the following, recall that we assume \(\int_{\mathbb{T}} f d\lambda =1.\) 

\begin{LEM}{\cite[Lemma 6.2]{Disjoint}}
\label{Denjoy-Koksma}
There exists a constant \(C(f)>0\) such that for every \(x\in \mathbb{T}\) and \(n\in \mathbb{N}\),
we have:
$$
|f^{(q_n)}(x)-q_n|<C(f)(\log(q_n)+|\log I_{n,x}|)
$$

$$
|f'^{(q_n)}(x)- (A_{-}-A_{+}) q_n\log (q_n)|<C(f)q_n(1+\frac{1}{I_{n,x}})
$$
where \(I_{n,x}:=q_n \displaystyle\min_{0\leq j< q_n}\|x+j\alpha\|\).
\end{LEM}
When \(m\in \mathbb{N}\) is arbitrary, we will use the Ostrowski expansion in \Cref{Classical  Ostrowski expansion} to iterate the above estimate whenever the pairs are away from singularities.
\begin{LEM}{\cite[Lemma 6.3]{Disjoint}}
\label{Sum Ostrowski}
There exists \(N_0>0\) such that for every \(N\geq N_0\) and \(x\in \mathbb{T}\) satisfying
\[\{x+j\alpha: 0\leq j \leq N\} \cap [-\frac{1}{2N\log ^4 N}, \frac{1}{2N \log ^4 N}]=\emptyset,\]
we have
\[|f^{(n)}(x)-n|<N^{1/5}\] for all \(n\in [0,N]\cap \mathbb{Z}.\)
\end{LEM}
\noindent If we denote \[O_n\coloneqq\{x\in \mathbb{T}: x-q_n\alpha,...,x,...,x+(q_n-1)\alpha\notin [-\frac{1}{2q_n \log^4q_n},\frac{1}{2q_n\log^4 q_n}]\}.\]
\noindent Given \(\e>0,\) there exists \(\N_0(\e)\), such that
\begin{equation}
\label{Uniform set O}
   O(\e)\coloneqq\displaystyle\bigcap_{n\geq \N_0(\e)} O_n 
\end{equation}
has measure \(\lambda(O(\e))\geq 1-\e\) and \(\mu^f(O(\e)^f)\geq 1-\e.\) \index{\(\N_0(\e)\)} \index{\(O(\e)\): Uniform set}This fact tells us that we can compare Birkhoff sums with flow time for \(x\) in an almost full measure set. We have the following derivative estimate to compare the divergence between the orbits of nearby points.  First, we need the orbits of points to be away from the singularity.
Denote 
\begin{equation}
\Sigma_{n}(M)\coloneqq \bigcup_{i=0}^{M q_{n+1}} R_{\alpha}^{-i}([-\frac{1}{q_n\log ^{7/8}q_n},\frac{1}{q_n\log ^{7/8}q_n}]).
\end{equation}
\begin{LEM}{\cite[Lemma 6.4]{Disjoint}}
\label{Denjoy Ostrowski}
Fix any \(M>1/16\). There exists \index{\(\e_0\): Error term in the Denjoy-Koksma estimate} \(\e_0=\e_0(f)\) such that for every \(0<\e<\e_0\) there exists \(m_0\coloneqq m_0(\e)\) such that for all \(n\geq m_0\), if \(\e^4q_n\leq m\leq Mq_{n+1}\) and \(x\notin \Sigma_{n}(M)\), we have 
\begin{equation}
((A_{-}-A_{+})-\e^2)m\log m \leq f'^{(m)}(x)  \leq ((A_{-}-A_{+})+\e^2)m\log m.
\end{equation}
Moreover, for every \(m\in[0,\e^4q_n]\cap \mathbb{Z},\) we have 
\begin{equation}
|f'^{(m)}(x)|<\e^2 q_n \log q_n. 
\end{equation}
\end{LEM}
\begin{RMK}
We slightly modify the statement of the previous lemma, but the proof is the same.
\end{RMK}
\noindent Later, we will pick \(M=\frac{1}{8}\), \(n=n_k\) where \(n_k\) is a large enough number in \Cref{Definition of rotation angle}. An important fact is that we can apply \Cref{Denjoy Ostrowski} to \(x\) in an almost full-measure set. Let
\[B_n\coloneqq\{x\in \mathbb{T}: x-q_n\alpha,...,x,...,x+(q_n-1)\alpha \notin (-\frac{1}{q_n \log^{7/8}q_n},\frac{1}{q_n \log^{7/8}q_n})\},\]
the asymmetry is related to definition \eqref{Definition of flow}.
\index{\(\N_1(\e)\)}
\noindent Given \(\e>0,\) there exists \(\N_1(\e)\in \mathbb{N}\), such that the following set
\index{\(Z(\e)\): Switchable Ratner set}
\begin{equation}
\label{Switchable Ratner set}
 Z(\e)\coloneqq \bigcap_{n\geq \N_1(\e),n\notin K_{\alpha}} B_n   
\end{equation}
has measure \(\lambda(Z(\e))\geq 1-\e\) and \(\mu^f(Z(\e)^f)\geq 1-\e\) where \(K_{\alpha}\) is as in \Cref{Definition of rotation angle}. We call \(Z(\e)\) and \(Z(\e)^f\)\textit{switchable Ratner sets}. In \cite{Multiple}, it is proved that for \(x\in Z(\e)\), at least one of the forward and backward orbits of \(x\) with a length greater than \(\frac{q_{n+1}}{8}\) is \(\frac{1}{q_n\log^{7/8}q_n}\) away from the singularity.
  \index{\(E_n\): The set such that orbits are away from singularities in both directions}
Denote \(E_{n}\) to be the set such that for every \(y\in E_{n}\),
\begin{equation}
 \left(\bigcup_{i=-4[q_{n}\log ^{\frac{1}{2}} q_{n}]}^{4[q_{n}\log ^{\frac{1}{2}}q_{n}]} R_{\alpha}^i y \right) \cap [-\frac{1}{q_{n}\log^{\frac{7}{8}} q_{n}},\frac{1}{q_{n}\log^{\frac{7}{8}} q_{n}}]=\emptyset.
\end{equation}
 \\
\(\textit{Observation 1:}\) When the distance between \(x\) and \(y\) in \eqref{Birkhoff differences} is small, we can apply the derivative estimate to obtain the divergence of Birkhoff sums. The derivative estimate applies even if \(x\) and \(y\) are in the same orbit.\\
\(\textit{Observation 2:}\) When the distance between \(x\) and \(y\) is arbitrary. If \(x=y+k\alpha\), by taking \(q_n\) much larger compared to \(k\), then the Birkhoff sum difference \(|f^{(q_n)}(x)-f^{(q_n)}(y)|\) is very small. When \(x\) and \(y\) are not in the same orbit, in \Cref{Small Shearing} we show that there is a sequence of \(\ell_nq_n\), such that the Birkhoff sum difference \(|f^{(\ell_nq_n)}(x)-f^{(\ell_nq_n)}(y)|\) is in a given compact set bounded away from 0 where \(\ell_n\) is some number less than \(\frac{q_{n+1}}{q_n}.\)


We will use the above observations to obtain the divergence of the Birkhoff sums, thereby completing the first step.
\subsection{Proof of Step 1}
Firstly, we will show the small shearing case where \((x,y)\) is a small pair or a close pair. For the small pair case, it uses the argument in \cite{Multiple} and \cite{Disjoint}. For the readers' convenience, we provide a proof here. The main difference is that we only need to show the divergence for special times and apply the argument for a pair of long-time segments. 
\index{\(\N_2(\e,\zeta)\)}
\begin{LEM}[Small Shearing]
\label{Small Shearing}
Let \(0<\e<\frac{\min\{H,\e_0\}}{100}\) and \(0<\zeta\leq\frac{\e}{10}\) be given. There exists  \(\N_2(\e,\zeta)\in \mathbb{N},\) such that if \((x,y)\) is a good pair of order \(k\), \(x\in Z(\e)\) and \(y\in E_{n_k}\) with \(k\geq \N_2(\e,\zeta)\), then there exist \(\ell_0\in \mathbb{Z}\), \(m\geq n_k\) and \(p\in P\) depending on \(x,y\) such that 
\begin{equation}
|f^{(\ell_0 q_m)}(x)-f^{(\ell_0 q_m)}(y)+p|\leq \zeta.
\end{equation}
Recall that \(P\subset \mathbb{R}\setminus \{0\}\) is defined in \eqref{The compact set} and \(n_k\) is in \Cref{Definition of rotation angle}.
\end{LEM}
\begin{proof}
There exists a unique \(m\in \mathbb{N},\) such that  \[\frac{1}{q_{m+1}\log q_{m+1}}\leq d_k<\frac{1}{q_{m}\log q_{m}}.\]
By definition, \(m\geq n_k.\) \\
\(\textit{Case 1:}\) when \((x,y)\) is a small pair, \(m\geq n_k+1.\)
Suppose \(q_{m+1} \geq 2 q_m\) first, then since \(x\in Z(\e),\) we know at least one of the following holds:
\begin{equation}
\label{Forward case a}
\tag{a}
\bigcup_{j=0}^{\max([\frac{q_{m+1}}{8},q_m])} R_{\alpha}^j x \cap [-\frac{1}{4q_m \log^{7/8}q_m},\frac{1}{4q_m \log^{7/8}q_m}]=\emptyset  
\end{equation}
and
\begin{equation}
\label{Backward case b}
\tag{b} 
\bigcup_{j=1}^{\max([\frac{q_{m+1}}{8},q_m])} R_{\alpha}^{-j} x \cap [-\frac{1}{4q_m \log^{7/8}q_m},\frac{1}{4q_m \log^{7/8}q_m}]=\emptyset.
\end{equation}
To see this, if \(m \in K_{\alpha},\) then by the definition of \(K_{\alpha}\), 
\(q_{m+1}<q_m\log^{\frac{7}{8}} q_m.\) There is at most one \(j_0\in[-\frac{q_{m+1}}{4}, \frac{q_{m+1}}{4}]\cap \mathbb{Z}\) so that \(x+j_0\alpha\in [-\frac{1}{4q_{m+1}}, \frac{1}{4q_{m+1}}].\) Note that \(\frac{1}{q_{m+1}}>\frac{1}{q_m \log^{7/8}q_m}.\) If \(j_0<0\), then \eqref{Forward case a} holds. If \(j_0\geq 0\) or does not exist, then \eqref{Backward case b} holds. If \(m\notin K_{\alpha},\) then we can use the definition of \(Z(\e)\):
\(x+j\alpha \notin(-\frac{1}{q_m \log^{7/8}q_m},\frac{1}{q_m \log^{7/8}q_m})\) for every \(-q_m\leq j\leq q_m-1.\) Let \(j_0\) be the index where \(\|x+j_0\alpha\|\) is the minimal among \(\|x+j\alpha\|\) where \(-q_m\leq j\leq q_m-1.\) If \(\|x+j_0\alpha\|<\|x+(j_0+q_m)\alpha\|,\) then we consider the backward orbit case. Note again there is at most one \(j\in [-q_m,-1]\cap \mathbb{Z}\) so that \(x+j\alpha\in [-\frac{1}{4q_m},\frac{1}{4q_m}].\) By iterating the backward orbits \([\frac{q_{m+1}}{8q_m}]\) times, then \eqref{Backward case b} holds. Similarly, if \(\|x+j_0\alpha\|\geq \|x+(j_0+q_m)\alpha\|,\) then \eqref{Forward case a} holds.\\
Assume without loss of generality, suppose the forward orbit case \eqref{Forward case a} is true.
Let \(j_0\) be the index such that \(\|x-(y+j_0\alpha)\|=d_k\) achieves the minimal distance. Assume \(j_0\geq 0\) first, then we have
\begin{equation}
\label{Basic estimate}
\begin{aligned}
    f^{(q_m)}(x)-f^{(q_m)}(y)=&f^{(q_m)}(x)-f^{(q_m)}(y+j_0\alpha)\\
    &+\sum_{i=0}^{j_0-1} f(y+i\alpha+q_m\alpha)-f(y+i\alpha).\\
\end{aligned}   
\end{equation}
If \eqref{Backward case b} holds, recall definition in \eqref{Definition of flow}, then we have
\begin{equation}
\begin{aligned}
    f^{(-q_m)}(x)-f^{(-q_m)}(y)=&\sum_{i=1}^{q_m}f(y+(j_0-1)\alpha-i\alpha)-f(x-i\alpha)\\
    &+\sum_{i=0}^{j_0-1} f(y+i\alpha)-f(y+i\alpha-q_m\alpha).\\
\end{aligned} 
\end{equation}
Let's denote \[I_1\coloneqq f^{(q_m)}(x)-f^{(q_m)}(y+j_0\alpha) \] and 
\[J_1\coloneqq \displaystyle \sum_{i=0}^{j_0-1} f(y+i\alpha+q_m\alpha)-f(y+i\alpha).\]
If \(j_0=0\), then we set \(J_1=0.\)
By \Cref{Denjoy-Koksma} and the mean value theorem, we have:
\begin{equation}
\label{Estimate of I_1}
(A_{-}-A_{+}-\e) q_m \log q_m d_k\leq |I_1
|\leq (A_{-}-A_{+}+\e) q_m \log q_m d_k.
\end{equation}
It might be the case that \(d_k\) is much smaller compared to \(\frac{1}{q_m \log q_m}\), however, we can use the SWR property to iterate the above estimate \eqref{Estimate of I_1} to obtain a difference away from zero.
That is, for every \(\ell\in \{1,2,..., \max\{\frac{q_{m+1}}{16q_m},1\}\}, \) we have
\begin{equation}
\begin{aligned}
    f^{(\ell q_m)}(x)-f^{(\ell q_m)}(y)=&f^{(\ell q_m)}(x)-f^{(\ell q_m)}(y+j_0\alpha)\\
    &+\sum_{i=0}^{j_0-1} f(y+i\alpha+\ell q_m\alpha)-f(y+i\alpha).\\
\end{aligned}   
\end{equation}
Similarly, the estimate for \(I_{\ell}\coloneqq f^{(\ell q_m)}(x)-f^{(\ell q_m)}(y+j_0\alpha) \) is:
\[(A_{-}-A_{+}-2\e) \ell q_m \log  q_m d_k\leq |I_{\ell}| \leq (A_{-}-A_{+}+2\e) \ell q_m \log  q_m d_k.\]
Therefore, there exists \(\ell_0\in \{1,2,..., \max\{\frac{q_{m+1}}{16q_m},1\}\}\), such that 
\begin{equation}
\label{Estimate of I}
    |I_{\ell_0}|\in[\frac{A_{-}-A_{+}}{64}, 2(A_{-}-A_{+})].
\end{equation}
 For the sum of the remaining terms \(J_{\ell}\coloneqq \displaystyle\sum_{i=0}^{j_0-1} f(y+i\alpha+\ell q_m\alpha)-f(y+i\alpha)\), recall from our assumption on set \(E_{n_k}\):  

\begin{equation}
\label{Bothside}
    \left(\bigcup_{i=0}^{j_0} R_{\alpha}^i [y,y+\ell q_m\alpha] \right) \cap [-\frac{1}{4q_{n_k} \log^{7/8} q_{n_k}},\frac{1}{4q_{n_k}\log ^{7/8} q_{n_k}}]=\emptyset.
\end{equation}
With above \eqref{Bothside}, the fact that \[\|\ell q_m\alpha\|\leq \frac{1}{q_{m}}\leq \frac{1}{q_{n_k+1}}\] and \(j_0\leq q_{n_k},\) we know the conditions for \Cref{Denjoy  Ostrowski} are satisfied. Therefore, for every \(\ell\in  \{1,2,..., \max\{\frac{q_{m+1}}{16q_m},1\}\}, \) there exists \(\theta_{\ell}\in[y,y+\ell q_m \alpha]\) such that 
\begin{equation}
\label{Estimate of J}
\begin{aligned}
    |J_{\ell}|&=\left|\displaystyle\sum_{i=0}^{j_0-1} f'(\theta_{\ell}+i\alpha)\|\ell q_m \alpha\|\right|\\
    &\leq  (H+1) j_0 \log j_0 \frac{1}{\r}\\
    &\leq \zeta/5
\end{aligned}
\end{equation}
as long as \(\frac{H+1}{\log \log (\tilde{N_2})} \leq \frac{\zeta}{10}.\)
\noindent Combine \eqref{Estimate of I} and \eqref{Estimate of J}, \(\exists \ell_0 \leq \max\{\frac{q_{m+1}}{16q_m},1\} \) and \(p\in P\) such that 
\begin{equation}
|f^{(\ell_0 q_m)}(x)-f^{(\ell_0 q_m)}(y)+p|\leq \zeta.
\end{equation}\\
When \(q_{m+1}< 2 q_m,\) then \[d_k \geq \frac{1}{2 q_m \log 2q_m} > \frac{1}{3q_m \log q_m}.\]
Then we replace the above proof by considering orbits of length \(q_{m-1}\) because \(q_{m+1}\geq 2 q_{m-1}\) and either forward orbits or backward orbits of length \(q_{m-1}\) are away from the singularity. Therefore, we complete the proof of this case.
When \(j_0<0\), then we can write
\begin{equation}
\begin{aligned}
    f^{(q_m)}(x)-f^{(q_m)}(y)=&f^{(q_m)}(x)-f^{(q_m)}(y+j_0\alpha)\\
    &+\sum_{i=j_0}^{-1} f(y+i\alpha+q_m\alpha)-f(y+i\alpha).\\
\end{aligned}   
\end{equation}\\
Iterating the above estimate, we obtain the following:
\begin{equation}
\begin{aligned}
    f^{(\ell q_m)}(x)-f^{(\ell q_m)}(y)=&f^{(\ell
     q_m)}(x)-f^{(\ell q_m)}(y+j_0\alpha)\\
    &+\sum_{i=j_0}^{-1} f(y+i\alpha+\ell q_m\alpha)-f(y+i\alpha).\\
\end{aligned}   
\end{equation}\\
The rest of the proof follows similarly as before.\\
\textit{Case 2:} when \((x,y)\) is a close pair, then we have \(\frac{5}{6\r } \leq d_k < \frac{1}{q_{n_k} \log q_{n_k}}\).\\
Suppose \(j_0\leq q_{n_k}/2,\) by \Cref{Denjoy  Ostrowski} and mean value theorem again, we have
\begin{equation}
(A_{-}-A_{+}-2\e)j_0 \log j_0 \|\ell q_{n_k} \alpha\|\leq |J_\ell|\leq (A_{-}-A_{+}+2\e)j_0\log j_0 \|\ell q_{n_k} \alpha\| 
\end{equation}
and 
\begin{equation}
(A_{-}-A_{+}-2\e) \ell q_{n_k} (\log \ell q_{n_k}) d_k\leq |I_{\ell}| \leq (A_{-}-A_{+}+2\e) \ell q_{n_k} (\log \ell q_{n_k}) d_k. 
\end{equation}
Observe that the growth rate of \(I_{\ell}\) and \(J_{\ell}\) are almost linear with respect to \(\ell\) because \(\ell \leq \frac{\r}{16q_{n_k}} \leq \log ^2q_{n_k}.\)
By our assumption on \(d_k\) and \(j_0\leq q_{n_k}/2\), we know \(|I_{\ell}| \geq \frac{5}{4} |J_{\ell}|\), the rest of proof is due to \cite[Lemma 4.8]{Multiple} again.\\
Suppose \(j_0> q_{n_k}/2.\) Let \(j_1=q_{n_k}-j_0\), consider 
\[J_{\ell}'=\sum_{i=1}^{j_1} f(y-i\alpha+\ell q_{n_k}\alpha)-f(y-i\alpha),\]
and 
\[I_{\ell}'=f^{(\ell q_{n_k})}(x)-f^{(\ell q_{n_k})}(y-j_1 \alpha).\]
Notice that \(\|x-(y-j_1\alpha)\| \geq d_k \geq \frac{5}{6q_{n_k+1}}\) and \(f^{(\ell q_{n_k})}(x)-f^{(\ell q_{n_k})}(y)= I_{\ell}'+J_{\ell}'.\) Since \(y\in E_{n_k}, \) as what we have shown before, we have 
\begin{equation}
(A_{-}-A_{+}-2\e)j_1 \log j_1 \|\ell q_{n_k} \alpha\|\leq |J_\ell'|\leq (A_{-}-A_{+}+2\e)j_1\log j_1 \|\ell q_{n_k} \alpha\| 
\end{equation}
and 
\begin{equation}
(A_{-}-A_{+}-2\e) \ell q_{n_k} (\log \ell q_{n_k}) \|x-(y-j_1\alpha)\|\leq |I_{\ell}'| \leq (A_{-}-A_{+}+2\e) \ell q_{n_k} (\log \ell q_{n_k}) \|x-(y-j_1\alpha)\|. 
\end{equation}
The above estimates imply that \(|I_{\ell}'|\geq \frac{5}{4}|J_{\ell}'|.\) So we complete the proof similarly as before.\\
\end{proof}
From \Cref{Small Shearing}, we can see the divergence of Birkhoff sums at time \(\ell_0q_m\).
We now claim that this estimate is preserved for relatively long orbits of \(x\) and \(y\), from which we can apply the ergodic theory later.
\begin{LEM}
\label{Small shearing preserving}
Under assumptions in \Cref{Small Shearing}. For \(i\leq |\ell_0 q_m|\) and \(|i-j|\leq 10 |\ell_0 q_m|^{1/5}\),  we have
\[|f^{(\ell_0 q_m)}(x+i\alpha)-f^{(\ell_0 q_m)}(y+j \alpha )+p|\leq \zeta, \hspace{5mm} \textit{\;if\;} \ell_0>0\]

or 
\[|f^{(\ell_0 q_m)}(x-i\alpha)-f^{(\ell_0 q_m)}(y-j \alpha )+p|\leq \zeta,  \hspace{5mm} \textit{\;if\;}  \ell_0<0.\]
\end{LEM}
\begin{proof} 
Since \(|i-j|\leq 10|\ell_0 q_m|^{1/5}\), which is much smaller compared to \(\ell_0 q_m\) , the proof of this lemma is similar to \Cref{Small Shearing} by comparing the growth rate of \(I_{\ell}\) and \(J_{\ell}\).
\end{proof}
From the small shearing estimate, we can control orbits of two points of \(x\) and \(y\) when \((x,y)\) is a good pair. We want to stress that we can only obtain the Birkhoff sums estimates for special times compared to the Birkhoff sums estimates in \cite{Multiple} and \cite{Disjoint}. However, we require the estimates for special times to be true for every point in a ``long'' orbit of \(x\) and \(y\), which follows from \Cref{Small shearing preserving}. We will use the small shearing estimates \Cref{Small Shearing} and \Cref{Small shearing preserving} in \Cref{Small shear flow} later. As mentioned before, the more difficult case is when \((x,y)\) is not a good pair, which makes the Birkhoff sum differences arbitrarily large. First, we want to show a good upper bound for \(d_k.\)
\begin{LEM}{\cite[Lemma 2.3]{Disjoint}}
\label{Space of orbits}
For any \(x\in\mathbb{T}\) and \(n\in \mathbb{N},\) the orbit \(\{x+i\alpha, 0\leq i <q_n\}\) is such that 
\[\|x+i\alpha-(x+j\alpha)\|\geq \frac{1}{2q_n},\]
for all \(0\leq i\neq j <q_n.\) In each interval of length \(\frac{2}{q_n}\), there must be a point of the form \(x+j\alpha\) for some \(0\leq j <q_n.\)
\end{LEM}

\begin{LEM}
\label{Distance estimate}
Suppose \(x\) and \(y\) are not on the same orbit, then \[0<d_k=\min\limits_{-q_{n_k}<j<q_{n_k}}\|x-R_{\alpha}^j y\|<\frac{5}{6q_{n_k}}\]
for every \(k.\)
\end{LEM}
\begin{proof}
   From Pigeonhole principle, there exists \(0< i_1< q_{n_k}\), such that \(\|x-(x+i_1\alpha)\|\leq \frac{1}{q_{n_k}}\).\\
   \textit{Case 1:} If the interval \([x,x+i_1\alpha]\subset \mathbb{T}\) contains some \(y+i_2\alpha\) for \(0\leq i_2 <q_{n_k}\), then \(\|x-(y+i_2\alpha)\|<\frac{1}{2q_{n_k}}\) or \(\|x+i_1 \alpha-(y+i_2\alpha)\|<\frac{1}{2q_{n_k}}\).\\
   \textit{Case 2:} If the interval \([x,x+i_1\alpha]\) doesn't contain any \(y+i_2\alpha\) for \(0\leq i_2 <q_{n_k}\). 
   Suppose \[\min\limits_{-q_{n_k}<j<q_{n_k}}\|x-R_{\alpha}^j y\|\geq\frac{5}{6q_{n_k}},\]
   then consider \(\{y+r\alpha\}_{r=0}^{q_{n_k}-1}.\) Let \(y+i\alpha\) be the element which is  closest to \(x\) and \(y+j\alpha\) be the element which is closest to \(x+i_1\alpha\). Then 
   \begin{equation}
   \begin{aligned}
       \|y+i\alpha-(y+j\alpha)\|&\geq \|y+i\alpha-x\|+\|x-(x+i_1\alpha)\|+\|x+i_1\alpha-(y+j\alpha)\|\\ &\geq \frac{5}{6 q_{n_k}}+\frac{1}{2q_{n_k}}+\frac{5}{6q_{n_k}}> \frac{2}{q_{n_k}}.  
   \end{aligned}
  \end{equation}
    The second inequality follows from \Cref{Space of orbits} and the fact that \(|j-i_1|<q_{n_k}\), so 
    \(\|x+i_1\alpha-(y+j\alpha)\|\geq \frac{5}{6q_{n_k}}\) as well. This is a contradiction because every interval of length at least \(\frac{2}{q_{n_k}}\) should contain \(y+k\alpha\) for some \(0\leq k< q_{n_k}\).
\end{proof}
To obtain the Birkhoff sum estimate when \((x,y)\) is not a good pair, we want to use the mean value theorem and the Denjoy-Koksma estimate in \Cref{Classical Denjoy Koksma} for the derivative of the roof function \(f\). However, we need to modify the argument because the interval may contain the singularity \(0\). Fortunately, the upper bound we get for \(d_k\) in \Cref{Distance estimate} enables us to obtain a good error term for the difference of Birkhoff sums. From the definition of good pair and \Cref{Distance estimate}, if \((x,y)\) is not a good pair of order \(k\), we have either
\begin{equation}
\label{Type I distance}
 \frac{1}{q_{n_k}\log q_{n_k}}\leq d_k\leq\frac{5}{6q_{n_k}}  
\end{equation}
or
\begin{equation} 
\label{Type II distance}
\frac{1}{q_{n_k+1}\log q_{n_k+1}}\leq d_k\leq \frac{5}{6q_{n_k+1}}. 
\end{equation}
\begin{DEF}
\label{Large shear type}
We say \((x,y)\) is type \(I\) of order \(k\) if \eqref{Type I distance} holds. We say \((x,y)\) is type \(II\) of order \(k\) if \eqref{Type II distance} holds.
\end{DEF}
\begin{LEM}[Large Shearing]
\index{\(\N_3(\e)\)}
\label{Large Shearing}
Let \(0<\e<\frac{\min\{H,\e_0\}}{100}\) be given. There exist constants \(\mathfrak{d}_1\),  \(\mathfrak{d}_2\) and \(\N_3(\e)\), such that if \(x\in E_{n_k}\cap E_{n_{k}+1}\) and \(y\in E_{n_k}\) with \(k\geq \N_3(\e)\),
we have
\begin{equation} 
\label{Forward large n_k}
\mathfrak{d}_1\leq|f^{(q_{n_k})}(x)-f^{(q_{n_k})}(y)|\leq \mathfrak{d}_2 \log q_{n_k},
\end{equation} 
when \((x,y)\) is type \(I\) of order \(k\),  
and
\begin{equation} 
\label{Forward large n_k+1}
\mathfrak{d}_1\leq|f^{(q_{n_k+1})}(x)-f^{(q_{n_k+1})}(y)|\leq \mathfrak{d}_2 \log q_{n_k+1},
\end{equation} 
when \((x,y)\) is type \(II\) of order \(k\).
\end{LEM}
\begin{proof}
For simplicity, denote \(m\coloneqq n_k\).
When \((x,y)\) is type \(I\) of order \(k\), we have  \(\frac{1}{q_m\log q_m}\leq d_k\leq \frac{5}{6q_m}. \) 
Since \(x\in E_{n_k}\), the forward orbit is bounded away from singularity by \(\frac{1}{q_m\log^{\frac{7}{8}} q_m}.\)
Let \(j_0\) be the index such that \(d_k=\|x-R_{\alpha}^{j_0}y\|\) achieves the minimal distance.
Let \(\mathcal{I}_{x,y}=\{i: 0\in [x+i\alpha,y+j_0\alpha+i\alpha], 0\leq i<q_m\}\).\\
\(\textit{Case 1:}\) \(\mathcal{I}_{x,y}=\emptyset.\) We then follow the same computation as in \Cref{Small Shearing}.
Recall we denote \[I_1\coloneqq f^{(q_m)}(x)-f^{(q_m)}(y+j_0\alpha) \] and 
\[J_1\coloneqq \displaystyle \sum_{i=0}^{j_0-1} f(y+i\alpha+q_m\alpha)-f(y+i\alpha).\]
Then we can write \(f^{(q_m)}(x)-f^{(q_m)}(y)=I_1+J_1.\)
Applying the same estimate shown in \Cref{Small Shearing}, we deduce that 
\(|J_1|\leq \e\) and 
\[H q_m\log q_m d_k \leq |I_1| \leq (H+1)q_m\log q_m d_k.\]
By our assumption in \eqref{Type I distance}, we will obtain 
\[\frac{99}{100} H\leq |I_1|-|J_1|\leq|f^{(q_{m})}(x)-f^{(q_{m})}(y)|\leq |I_1|+|J_1|\leq (H+1) \log q_{m}.\]
\noindent\(\textit{Case 2:}\) \(\mathcal{I}_{x,y}\neq \emptyset.\) We claim that the cardinality of \(\mathcal{I}_{x,y}\) is no more than 2.
Suppose that there are three elements in \(\mathcal{I}_{x,y}\), we assume without loss of generality that \(0\in \mathcal{I}_{x,y}\). Suppose \(i_1,i_2 \in \mathcal{I}_{x,y}\) and they are nonzero, then \(y+j_0\alpha+i_1\alpha\), \(y+j_0\alpha+i_2\alpha \in [x,y+j_0\alpha]\) by the definition of \(j_0.\) From \Cref{Space of orbits}, either \(\|i_2\alpha\|\) or \(\|i_1 \alpha\| \geq \frac{1}{q_{n_k}}.\) This is a contradiction to the fact that \(\|x-(y+j_0\alpha)\|\leq \frac{5}{6 q_{n_k}}.\) Therefore, we can divide the term \(I_1\) into two parts:
\begin{equation}
\label{type I estimate}
\begin{aligned}
      &\left||\sum_{j=0}^{q_m-1}\bar{f}_j(x,y)|-|\sum_{i\in \mathcal{I}_{x,y}}f(x+i\alpha)-f(y+j_0\alpha+i\alpha)|\right|\\
      &\leq |I_1|\\
      &\leq \left||\sum_{j=0}^{q_m-1}\bar{f}_j(x,y)|+|\sum_{i\in \mathcal{I}_{x,y}}f(x+i\alpha)-f(y+j_0\alpha+i\alpha)|\right|,
  \end{aligned}  
\end{equation}
where 
\begin{equation}
\bar{f}_j(x,y)=\left\{
\begin{aligned}
&f(x+j\alpha)-f(y+j_0\alpha+j\alpha)  & if \; j\notin \mathcal{I}_{x,y},\\
&0   &  otherwise. \\
\end{aligned}
\right.
\end{equation}
Apply the mean value theorem to \(\sum_{i=0}^{q_m-1}\bar{f}_i(x,y)\), there exists \(\theta \in [x,y]\) such that 
\begin{equation}
\left|\sum_{i=0}^{q_m-1}\bar{f}_i(x,y)\right|
 \leq\left(\left|\sum_{i=0}^{q_m-1}\bar{f_m}'(\theta+i\alpha)\right|+4A_{-}q_m\log^{7/8}q_m\right)d_k,   
\end{equation}
where 
\begin{equation}
\bar{f_m}(x)=\left\{
\begin{aligned}
&f(x)  & if \; \|x\|\geq \frac{1}{q_m\log^{7/8} q_m},\\
&0   &  otherwise. \\
\end{aligned}
\right.
\end{equation}
Following the proof of \Cref{Denjoy-Koksma}, apply \Cref{Classical Denjoy Koksma} to \(\bar{f_m}'\), we have 
\[\left|\sum_{i=0}^{q_m-1}\bar{f_m}'(\theta+i\alpha)-q_m \int_{\mathbb{T}} \bar{f_m}'d\lambda \right|<2Var(\bar{f_m}'),\]
where \(\lambda\) is the Haar measure on \(\mathbb{T}.\)
Note that \(f'(x)=-\frac{A_{-}}{x}+\frac{A_{+}}{1-x}+g'(x).\) Let \(v_m=\frac{1}{q_m\log^{7/8}q_m}\), therefore,
\[\int_{\mathbb{T}} \bar{f_m}'d\lambda=\int_{v_m}^{1-v_m} f' d\lambda=f(1-v_m)-f(v_m).\]
When \(m\) is large enough:
\[|f(1-v_m)-f(v_m)| \leq (A_{-}-A_{+}+\frac{H}{200})\log q_m,\]
and
\[Var(\bar{f_m}') \leq \frac{H}{1000} q_m \log q_m.\]
Combining the above two estimates, we obtain:
\[\left|\sum_{i=0}^{q_m-1}\bar{f_m}'(\theta+i\alpha)\right|\leq (A_{-}-A_{+}+\frac{H}{100})q_m\log q_m.\]
Hence,
\begin{equation}
\left|\sum_{i=0}^{q_m-1}\bar{f}_i(x,y)\right| \leq (A_{-}-A_{+}+\frac{H}{50})q_m\log q_md_k.   
\end{equation}
Since \(\frac{1}{q_m\log q_m}\leq d_k\leq \frac{5}{6q_m},\) the above estimate implies 
\begin{equation}
\label{Middle Large} 
 \left|\sum_{i=0}^{q_m-1}\bar{f}_i(x,y)\right| \leq \frac{5}{6}(A_{-}-A_{+}+\frac{H}{50})\log q_m.
\end{equation}
Then, it suffices to give an estimate of the pair whose interval contains the singularity \(0.\) For every \(i\in \mathcal{I}_{x,y}\), we have
\begin{equation}
\label{Upper Large}
\begin{aligned}
  |f(x+i\alpha)-f(y+j_0\alpha+i\alpha)|&\leq A_{-} \log (q_m \log^{\frac{7}{8}} q_m)\\
  &-A_{+} \log (q_m) +\frac{H}{100} \log (q_m)\\
  &\leq (A_{-}-A_{+}+\frac{H}{50})\log (q_m),
\end{aligned} 
\end{equation}
and
\begin{equation}
\label{Lower Large}
\begin{aligned}
  |f(x+i\alpha)-f(y+j_0\alpha+i\alpha)|&\geq A_{-} \log (q_m)\\
  &-A_{+} \log (q_m \log q_m) -\frac{H}{100} \log (q_m)\\
  &\geq (A_{-}-A_{+}-\frac{H}{50})\log (q_m).
\end{aligned}  
\end{equation}
Combine \eqref{Middle Large}, \eqref{Lower Large} and \eqref{Upper Large}, we will obtain:
\[\left(\frac{1}{6}(A_{-}-A_{+})-\frac{H}{25}\right) \log(q_m)\leq|f^{(q_m)}(x)-f^{(q_m)}(y)|\leq 2(A_{-}-A_{+}+\frac{H}{50}) \log(q_m).\]
Take \(\mathfrak{d}_1=\frac{99H}{100}\) and \(\mathfrak{d}_2=2(A_{-}-A_{+}+\frac{H}{50})\), we complete the proof of the first case.\\
When \((x,y)\) is type \(II\) of order \(k\), we have \(\frac{1}{q_{m+1} \log q_{m+1}} < d_k <\frac{5}{6 q_{m+1}}.\)  We only need to replace \(m\) in all previous estimates with \(m+1\). Thus, we complete the proof.
\end{proof}

\subsection{Step 2}The second step uses the divergence of Birkhoff sums to classify the centralizer of the special flow. Suppose \(\nu\) is 2-simple, recall from \eqref{Carrying set of graph joining}, we can pick a Lusin set such that \(\phi\) is uniformly continuous on that subset. Recall from \eqref{Definition of flow},  \[T_{t}^f(x,s)=(x+m\alpha,t+s-f^{(m)}(x)).\] 
In particular, taking \(t=f^{(q_n)}(x)\), we have
\(T_{t}^f(x,s)=(x+q_n\alpha,s)\) when restricted to a suitable subset. Suppose that \(\phi\) is a non-trivial centralizer, \((x,s)\) and \(T_{t}^f(x,s)\) are in the Lusin set.
In the small shearing case, by the divergence of Birkhoff sums, the distance between \(\phi(x,s)\) and 
\(T_{t}^f\phi(x,s)\) will stay within a compact set by \Cref{Small Shearing}, this is a contradiction to the uniform continuity of \(\phi\). The more interesting part is the large shearing case. Due to the existence of singularities, the shearing effect is very strong, then \(\phi(x,s)\) and 
\(T_{t}^f\phi(x,s)\) may still be very close to each other after shifting a large distance. However, the horizontal distance between \(\phi(x,s)\) and 
\(T_{t}^f\phi(x,s)\) will be much larger compared to \(\|q_n\alpha\|\) then. Therefore, by comparing the orbits of the pair and applying the uniform Birkhoff estimate in \Cref{Classical uniform Birkhoff ergodic theorem}, we can conclude the trivialness of centralizers.
The proof of the following uniform Birkhoff ergodic theorem is an application of Birkhoff ergodic theorem and Egorov's theorem. 
\begin{LEM}
\label{Classical uniform Birkhoff ergodic theorem}
Let \(\{T_t\}_{t\in \mathbb{R}}\) be an ergodic flow acting on \((X,\mu)\) and \(\rho\) be an ergodic self-joining of \((X,(T_t)_{t\in \mathbb{R}},\mu).\) For measurable sets \(A, B\) and every \(\eta, \theta ,\kappa>0\), there exist \(N=N(\eta,\theta,\kappa)\) and a set \(\mathcal{E}=\mathcal{E}(\eta,\theta,\kappa)\) with \(\rho(\mathcal{E})>1-\theta\) such that for every \(M,L \geq N\) with \(L \geq \kappa M\), we have
\[\left|\frac{1}{L} \int_{M}^{M+L} \mathds{1}_{A \times B}(T_tx,T_ty)dt-\rho(A\times B)\right|<\eta\]
and
\[\left|\frac{1}{L} \int_{M}^{M+L} \mathds{1}_{A \times B}(T_{-t}x,T_{-t}y)dt-\rho(A\times B)\right|<\eta\]
for every \((x,y)\in \mathcal{E}.\)
\end{LEM}
\subsection{Proof of Step 2}
Given \(\alpha \in \mathcal{ D}\), let \(\nu\) be an ergodic self-joining of the special flow.
\index{\(G_k\): The set of good pairs of order \(k\)}
\index{\(\hat{H}\): The minimal shifting distance within the compact set} 
Define 
\begin{equation}
\label{Definition of G_k}
  G_k\coloneqq\{((x,s_1),(y,s_2))\in X^f\times X^f: (x,y) \textit{\;is a good pair of order k}\},   
\end{equation}
and
\[\hat{H}\coloneqq \min \{\frac{H}{200},\min_{1\leq k \leq \frac{|2H+3|}{a}+1}{\|k \alpha\|}\}. \]
For given \(\e>0,\) we define
 \index{\(\delta\): Distance of nearby points}
\begin{equation}
\label{Value of delta}
   \delta \coloneqq \frac{1}{100}\min \{\hat{H}, \frac{\e^2}{10}\},
\end{equation}
and
\index{\(X_{\e}\): Truncated space}
\begin{equation}
\label{Set X_{e}}  
X_{\e}\coloneqq \{(x,s) \in X^f: \|x\|\geq \frac{\e}{\log^2(\e)}, \frac{\e}{100}\leq s \leq f(x)-\frac{\e}{100}\}.
\end{equation}
From our assumption on the roof function \eqref{Roof function},
\index{\(\N_4(\e)\)}
there exists \(\N_4(\e),\) such that for all \(x',y'\in \mathbb{T}\), \(n\geq \N_4(\e)\) with \(\|x'\|\geq \frac{\e}{\log^2(\e)}\) and \(\|x'-y'\|\leq \frac{1}{q_{n}}\), we have
\begin{equation}
\label{distance of fiber}
    |f(x')-f(y')|\leq \delta.
\end{equation}
By taking \(\e\) small enough, we can assume that for every \(p\in P\) and \((x,s), (x',s')\in X^f\) with \(d((x,s),(x',s'))\leq \delta\) and \(T_{p}^f(x,s)\in X_{\e},\) then 
\begin{equation}
\label{Continuity of nearby points within distance P}
d(T_{p}^f(x,s),T_{p}^f(x',s'))\leq \frac{\e}{1000}.
\end{equation}
For \((x,s)\in X^f,\) and \(t\in\mathbb{R}\)
define \(n(x,s,t)\) and \(s_t\) as unique numbers such that \index{\(n(x,s,t)\): The number occurs in orbits of the special flow}
\index{\(s_t\): The second coordinate of the special flow}
\begin{equation}
\label{Coordinates of special flow}
 T_t^f(x,s)=(x+n(x,s,t)\alpha,s_t).   
\end{equation}
\begin{LEM}
\label{Small shear flow}
Under assumptions in \Cref{Theorem of centralizer}, suppose \(\nu\) is an ergodic 2-simple self-joining of \((X^f,(T_t^f)_{t\in \mathbb{R}}, \mu^f)\) and \(\nu(\Omega
)=1\) with \(\Omega\) in \eqref{Carrying set of graph joining}. Given \(0<\e<\frac{\min\{\hat{H},\e_0\}}{100}\), suppose that \(\nu(G_k)>\e\) for large enough \(n_k\) in \Cref{Definition of rotation angle}, then \(\phi\) is a trivial centralizer.
\end{LEM}
 \begin{proof}
Assume \(\phi\) is not a trivial centralizer. For the given \(\e>0\), there is a subset \(\tilde{U}\subset X^f\) with \(\mu^f(\tilde{U})>1-\frac{\e}{100}\), such that \(\phi\) is uniformly continuous on \(\tilde{U}.\)  If \((x,s_1), (y,s_2) \in \tilde{U}\) and \(\|x-y\|+|s_1-s_2|\leq \tau\), then 
\[d\left(\phi(x,s_1),\phi(y,s_2)\right) \leq \hat{\delta}(\tau).\]
\index{\(\tau\): Distance between base points}There exists \(\tau_0\), such that if \(\tau<\tau_0\), then
\( \hat{\delta}(\tau)<\delta.\) 
By taking \(\e\) small enough,
\[\mu^f(X_{\e})\geq 1-\frac{\e}{40}.\]
Let \(\mathcal{X}(\e)\) be \(\tilde{U}\cap X_{\e}\), then \(\mu^f(\mathcal{X}(\e))\geq 1-\frac{\e}{20}.\)
Therefore,
\[\nu(\mathcal{X}(\e)\times \mathcal{X}(\e))\geq 1-\frac{\e}{10}.\]
By ergodicity of \(T^f\times T^f,\) there exist a set \(\mathcal{U}(\e)\) with \(\nu(\mathcal{U}(\e))\geq 1-\frac{\e}{100}\) and \(T_0(\e)>0,\) such that for every \(T'\geq T_0(\e)\) and every \(((x,s_1),(y,s_2))\in \mathcal{U}(\e)\), we have
\[\left|\frac{1}{T'}\int_0^{T'} \mathds{1}_{\mathcal{X}(\e)\times \mathcal{X}(\e)}(T_t^f(x,s_1), T_t^f(y,s_2)) d t\right|\geq 1-\frac{\e}{5}.\]
Recall that \(Z(\e)\) is the set satisfying the SWR property in \eqref{Switchable Ratner set} and \( O(\e)\) is the set in \eqref{Uniform set O}. Pick \(n_k\) large enough such that \[n_k\geq \max \{\N_0(\frac{\e}{100}),\N_1(\frac{\e}{100}),\N_2(\frac{\e}{100},\frac{\e}{1000}),\N_3(\frac{\e}{100}),\N_4(\e)\}\] and \(q_{n_k} \geq \max\{T_0(\e)/2a,\frac{1}{\tau_0^3},\frac{1}{\delta^3}, 10^8\}. \)
Denote
\[\tilde{B}\coloneqq G_k\cap \mathcal{U}(\e)\cap \left(\left(Z(\frac{\e}{100})\cap O(\frac{\e}{100})\right)^f \times E_{n_k}^f\right).\]
For all \(n_k\) large enough, 
\(\nu(\tilde{B})>0.\)
Pick \(((x,s),(y,s'))\in \tilde{B} \), we know the assumption for \Cref{Small Shearing} is satisfied.
By \Cref{Small Shearing}, there exist \(\ell_0\), \(m\) and \(p\in P\) depending on \(x,y\) such that 
\[|f^{(\ell_0 q_m)}(x)-f^{(\ell_0 q_m)}(y)+p|\leq \frac{\e}{1000}.\]
Consider the orbits \(\{T_t^f(x,s): t\in[0, f^{(2\ell_0 q_m)}(x)]\} \).
Since \(((x,s)),(y,s'))\in \mathcal{U}(\e),\) the set \(\Lambda=\{t\leq f^{(2\ell_0 q_m)}(x) : T_t^f(x,s), T_t^f(y,s') \in \mathcal{X}(\e)\}\) has density at least \(1-\frac{\e}{5}.\)\\
For each \(t\leq f^{([0.99\ell_0 q_m])}(x)\), let 
\[x_t=x+n_1\alpha \text{\;\;and\;\;} y_t=y+n_2\alpha\]
where \(n_1=n(x,s,t)\) and \(n_2=n(y,s',t)\) defined as in \eqref{Coordinates of special flow}. For \(t\leq f^{([0.99\ell_0 q_m])}(x)\), by \Cref{Sum Ostrowski}, we have 
\begin{equation}
|n_1-n_2|\leq 2 (\ell_0 q_m)^{1/5}.
\end{equation}
From \Cref{Small shearing preserving}, for every \(0\leq n_1\leq \ell_0 q_m\) and corresponding \(n_2\) (assuming forward orbit case), we have:
\begin{equation}
\label{Density argument}
   |f^{(\ell_0 q_m)}(x+n_1\alpha)-f^{(\ell_0 q_m)}(y+n_2\alpha)+p|\leq \frac{\e}{1000}. 
\end{equation}
For each \(t \in \Lambda,\) with \(t\leq f^{([0.99\ell_0 q_m])}(x)\), consider \[(\tilde{x}_t,\tilde{s}_{t})=T_{f^{(\ell_0 q_m)}(x_t)}^f(x_t,s_t),\]
where \(s_t\) is defined as in \eqref{Coordinates of special flow} and similarly for \(s'_t.\) Suppose there is set \(\Lambda'\subset [0, f^{([0.99\ell_0 q_m])}(x)]\) with density larger than \(\e\) such that for every \(t\in \Lambda',\) either \((x_t,s_t)\) or \((\tilde{x}_t,\tilde{s}_t)\) is not in \(\mathcal{X}(\e).\) Then we can obtain a contradiction to \(\Lambda.\)  However, if  \((x_t,s_t)\) and \((\tilde{x}_t,\tilde{s}_t)\) are in \(\mathcal{X}(\e)\), and \(|x_t-\tilde{x}_t|\leq \|\ell_0 q_m \alpha\|\leq \frac{1}{\sqrt{q_{n_k}}} < \tau_0\), then we know \(d(\phi(x_t,s_t),\phi(\tilde{x}_t,\tilde{s}_t))\leq \delta.\)
Because \(\phi\) commutes with \(T_t^f,\)
\[\phi(\tilde{x}_t,\tilde{s}_t)=T_{f^{(\ell_0q_m)}(x_t)}^f\phi(x_t,s_t)=T_{f^{(\ell_0q_m)}(x_t)}^f(y_t,s'_t).\]
Hence,
\[d((y_t,s'_t),T_{f^{(\ell_0q_m)}(x_t)}^f(y_t,s'_t)) \leq \delta.\]
From \eqref{distance of fiber} and the fact that \((y_t,s'_t)\in \mathcal{X}(\e),\)
\[d((y_t,s'_t),T_{f^{(\ell_0q_m)}(y_t)}^f(y_t,s'_t))\leq 2 \delta.\]
From \eqref{Density argument}, 
\[d(T_{f^{(\ell_0q_m)}(x_t)}^f(y_t,s'_t),T_{f^{(\ell_0q_m)}(y_t)-p}^f(y_t,s'_t))\leq \frac{\e}{1000}.\]
This is a contradiction to our choice of \(\delta\) in  \eqref{Value of delta} and \(\e\) because the above three estimates imply that 
\[\frac{\e}{100}\leq\hat{H}\leq d(T_{f^{(\ell_0q_m)}(y_t)}^f(y_t,s'_t),T_{f^{(\ell_0q_m)}(y_t)-p}^f(y_t,s'_t))\leq 3\delta+ \frac{\e}{1000}<\frac{\e}{200}.\]
 \end{proof}
By \Cref{Small shear flow}, the pair of points satisfying the small shearing case in \Cref{Small Shearing} consists of a set of small measure. Now we focus on the large shearing case. For pairs \((x,s)\) and \((y,s')\) in a ``good'' set, \[T_{f^{(q_r)}(x)}^f(x,s)=(x+q_r\alpha,s_t)\] and \[T_{f^{(q_r)}(x)}^f(y,s')=(y+D\alpha,s'_t)\] for some integer \(D\) depending on \(x,y,r,s'\) where \(r\) is either \(n_k\) or \(n_k+1.\) The only possibility that  \(T_{f^{(q_r)}(x)}^f(y,s')\) in the \(\delta\) neighborhood of \((y,s')\) is that \(\|D\alpha\| \leq \delta\). Moreover, \(\|D\alpha\|\geq \frac{1}{q_{n_k}}\) because of large shearing estimate. However, \(\|q_r\alpha\| \leq \frac{1}{q_{n_k+1}}\), so we can see the divergence of nearby orbits because of this difference.

\begin{proof}[Proof of \Cref{Theorem of centralizer}]
Assume there is a non-trivial centralizer \(\phi: X^f \rightarrow X^f.\) Let \(\nu\) be the ergodic self-joining such that \(\nu(\Omega)=1\) where \(\Omega\) is in \eqref{Carrying set of graph joining}. Let \(n_k\) be a number as in \Cref{Definition of rotation angle} which will be determined later and \(0<\e<\frac{\min\{\hat{H},\e_0\}}{100}.\) Recall the definition of \(G_k\) in \eqref{Definition of G_k}, \(X_{\e}\) in \eqref{Set X_{e}} and \Cref{Large shear type}. We can find a subset \(\tilde{U}\subset X^f\) with \(\mu^f(\tilde{U})>1-\frac{\e}{100}\), such that \(\phi\) is uniformly continuous on \(\tilde{U}.\) Moreover, there exists \(\tilde{\delta}(\e)\) such that when \((x,s_1)\) and \((y,s_2)\in \tilde{U}\) and \(d((x,s_1),(y,s_2))\leq \tilde{\delta}\), then \(d(\phi(x,s_1),\phi(y,s_2))\leq \e.\) Denote \(\mathcal{X}(\e)\coloneqq\tilde{U}\cap X_{\e}.\) Apply \Cref{Classical uniform Birkhoff ergodic theorem}
to set \(\mathcal{X}(\e)\times \mathcal{X}(\e)\), \(\eta=\frac{\e}{100}\), \(\theta=\frac{\e}{100}\) and \(\kappa\coloneqq \kappa(\e)\) which will be determined later, we can find a set \(\mathcal{E}\coloneqq\mathcal{E}(\eta,\theta,\kappa)\) with \(\nu(\mathcal{E})\geq 1-\frac{\e}{100}\) and a number \(N\coloneqq N(\eta,\theta,\kappa)\)
such that for every \(M,L \geq N\) with \(L \geq \kappa M\), we have
\begin{equation}
\label{Forward classical uniform Birkhoff estimate}
  \left|\frac{1}{L} \int_{M}^{M+L} \mathds{1}_{\mathcal{X}(\e) \times \mathcal{X}(\e)}(T_t^f(x,s_1),T_t^f(y,s_2))dt-\nu(\mathcal{X}(\e)\times\mathcal{X}(\e))\right|<\eta  
\end{equation}
and
\begin{equation}
\label{Backward classical uniform Birkhoff estimate} 
\left|\frac{1}{L} \int_{M}^{M+L} \mathds{1}_{\mathcal{X}(\e) \times \mathcal{X}(\e)}(T_{-t}^f(x,s_1),T_{-t}^f(y,s_2))dt-\nu(\mathcal{X}(\e)\times \mathcal{X}(\e))\right|<\eta
\end{equation}
for every \(((x,s_1),(y,s_2))\in \mathcal{E}.\) For large enough \(n_k\), observe that if \(((x,s_1),(y,s_2))\in  G_k^c\), then \((x,y)\) is either type \(I\) or type \(II\) of order \(k\) by \Cref{Distance estimate}. Define 
\begin{equation}
\label{nearby points}
  (x',s_1)\coloneqq T_{f^{(q_r)}(x)}^f\x \textit{\;\;and\;\;} (y',s_2')\coloneqq T_{f^{(q_r)}(x)}^f\y,  
\end{equation}
where \(r=n_k\) if \((x,y)\) is type \(I\) of order \(k\) and \(r=n_{k}+1\) if \((x,y)\) is type \(II\) of order \(k.\) If \(((x,s_1),(y,s_2)) \in G_k^c\) and satisfying conditions in \Cref{Large Shearing}, we can find \(\mathfrak{d}_1\) and \(\mathfrak{d}_2\) such that 
\begin{equation}
\label{n_k condition 1}
\mathfrak{d}_1\leq|f^{(q_r)}(x)-f^{(q_r)}(y)|\leq \mathfrak{d}_2 \log(q_r)
\end{equation}
when \(n_k\) is large enough. 
Suppose \((x,s_1)\), \((x',s_1)\in \tilde{U}\),
there is \(\ell\in \mathbb{N}\), such that \[\frac{1}{q_{\ell+1}\log q_{\ell+1}}\leq\|y-y'\|\leq \frac{1}{q_{\ell}\log q_{\ell}},\]
where \(\frac{2}{q_{\ell}\log q_{\ell}} \geq \displaystyle \min \{\|j\alpha\|: 1\leq j \leq \frac{2\mathfrak{d}_2\log q_r}{a}\}\) and \(a\) is the lower bound of the roof function. Take \(n_k\) large enough so that 
\begin{equation}
\label{n_k condition 2}
q_{\ell} \geq \frac{2N}{a}.
\end{equation}  Note that for different \(x\), \(\ell\) may be different. However, the lower bound is uniform because of uniform continuity. Recall that \(Z(\e)\) is the set satisfying the SWR property in \eqref{Switchable Ratner set} and \( O(\e)\) is the set in \eqref{Uniform set O}. Define \(S_{n}: X^f \rightarrow X^f\) to be a map such that
\[S_n(x,s)=(x+q_n\alpha,s)\]
and \(\pi_1:X^f\times X^f \rightarrow X^f\) to be a map such that
\[\pi_1((x,s_1),(y,s_2))=(x,s_1).\]
Consider sets \[\mathcal{Q}_k\coloneqq \mathcal{X}(\e) \cap \pi_1 \mathcal{E}\cap Z(\frac{\e}{100})^f\cap O(\frac{\e}{100})^f\cap E_{n_k}^f\cap E_{n_k+1}^f\] and
\[\mathcal{U}_k=(\mathcal{Q}_k \cap S_{n_k}^{-1} \mathcal{Q}_k \cap S_{n_k+1}^{-1} \mathcal{Q}_k) \times \left(E_{n_k}^f \cap Z(\frac{\e}{100})^f\cap O(\frac{\e}{100})^f\right).\]
Here, we choose our \(n_k\) so that 
\eqref{n_k condition 1}, \eqref{n_k condition 2}, \Cref{Small shear flow} hold simultaneously.
Moreover, \(\nu(\mathcal{U}_k\cap G_k^c)>0,\) and \(\log \log q_{n_k}\geq \frac{2}{\tilde{\delta}}.\) These can be guaranteed when \(n_k\) is large enough. For such \(n_k\), pick \(((x,s_1),(y,s_2))\in \mathcal{U}_k\cap G_k^c\) and consider \(((x',s_1),(y',s_2'))\) defined in \eqref{nearby points}. Let \(\e'\) be \(\frac{\e}{1000}\). Since \(\left((x',s_1),(y',s_2')\right)\in \mathcal{Q}_k \times \left(E_{n_k}^f \cap Z(\frac{\e
}{100})^f\cap O(\frac{\e}{100})^f\right),\) 
by the proof of \cite[Proposition 4.5]{Multiple}, there exist  \(\kappa'\coloneqq \kappa'(\e')\), \(n_0\) and \(p\in P\) as in \eqref{The compact set} depending on \(y\) and \(y'\), such that 
\begin{equation}
\label{Denjoy estimate for fiber}
    f^{(n_0q_{\ell})}(y)-f^{(n_0q_{\ell})}(y')=p
\end{equation}
and
\begin{equation}
\label{Cocycle identity for fiber}
|f^{(i)}(y+n_0q_{\ell}\alpha)-f^{(i)}(y'+n_0q_{\ell}\alpha)|\leq \e'.
\end{equation}
for every \(i\in [1, [\kappa'n_0q_{\ell}]+1]\) if \(n_0>0\) or \(-i\in [1, [\kappa'n_0q_{\ell}]+1]\) if \(n_0<0.\)
Without loss of generality, we will only consider the forward orbit case, that is, \(n_0>0\). 
Applying \Cref{Denjoy Ostrowski} and mean value theorem, for every \(0\leq i \leq 2n_0 q_{\ell},\) we have 
\begin{equation}
\label{Denjoy estimate for base}
    |f^{(i)}(x)-f^{(i)}(x')|\leq \e'q_{n_k}\log(q_{n_k})\|q_r\alpha\|\leq \frac{\tilde{\delta}}{2}
\end{equation}
because \(\|q_r\alpha\|\leq \frac{1}{q_{n_k+1}} \leq \frac{1}{q_{n_k}\log q_{n_k} \log \log q_{n_k}}.\) Our \(\kappa\) is defined as \(\kappa=\frac{\kappa'}{10}.\) 
Pick \(M=f^{(n_0q_{\ell})}(y)\) and \(L=\kappa M.\) Since \((x,s_1)\) and \((x',s_1)\) are in the set \(\pi_1\mathcal{E}\), we can apply \eqref{Forward classical uniform Birkhoff estimate}. By our choice of  \(\e\), \(\eta\) and \(\theta\), we can pick a time \(t_{*}\in[M,M+L]\) so that \(T_{t_{*}}^f(x,s_1)\) and \(T_{t_{*}}^f(x',s_1)\in \mathcal{X}_{\e}.\)
By definition, there exist \(n_1= n(x,s_1,t_{*})\) and \(n_2= n(y,s_2,t_{*})\in \mathbb{N}\) such that
\begin{equation}
\label{t* coordinate of x}
T_{t_{*}}^f(x,s_1)=(x+n_1\alpha,s_1+t_{*}-f^{(n_1)}(x))
\end{equation}
and
\begin{equation}
\label{t* coordinate of y}
T_{t_{*}}^f(y,s_2)=(y+n_2\alpha,s_2+t_{*}-f^{(n_2)}(y)).
\end{equation}
By \eqref{Denjoy estimate for base}, we have 
\begin{equation}
\label{t* coordinate of x'}
T_{t_{*}}^f(x',s_1)=(x'+n_1\alpha,s_1+t_{*}-f^{(n_1)}(x'))    
\end{equation}
and 
\begin{equation}
\label{t* distance}
d(T_{t_{*}}^f(x,s_1), T_{t_{*}}^f(x',s_1))\leq \tilde{\delta}.  \end{equation}
Therefore, by \eqref{t* distance} and uniform continuity of \(\phi\), 
\begin{equation}
\label{t* distance between y and y'}
d(T_{t_{*}}^f(y,s_2), T_{t_{*}}^f(y',s_2'))=d(\phi(T_{t_{*}}^f(x,s_1)),\phi(T_{t_{*}}^f(x',s_1))) \leq \e.
\end{equation}
However, by \eqref{Denjoy estimate for fiber} and \eqref{Cocycle identity for fiber}, we have 
\begin{equation}
\label{t* coordinate of y'}
T_{t_{*}-p}^f(y',s_2')=(y'+n_2\alpha,s_2'+t_{*}-p-f^{(n_2)}(y')).
\end{equation}
Combine \eqref{t* coordinate of y} and \eqref{t* coordinate of y'}, we obtain 
\begin{equation}
\begin{aligned}
  d(T_{t_{*}}^f(y',s_2'),T_{t_{*}}^f(y,s_2))&\geq d(T_{t_{*}-p}^f(y',s_2'), T_{t_{*}}^f(y',s_2'))-d(T_{t_{*}-p}^f(y',s_2'),T_{t_{*}}^f(y,s_2))\\& \geq \hat{H}-3\e  \geq  50\e,  
\end{aligned}
\end{equation}
which is a contradiction to \eqref{t* distance between y and y'}.
\end{proof}

\section{\textbf{General case of self-joining}}
\label{General case of self-joining}
In this section, we prove \Cref{Main}. We show that if  \(\alpha\in \mathcal{D}\) in \Cref{Definition of rotation angle} and \(f\) satisfies the condition in \Cref{Definition of roof}, then the special flow has the minimal self-joining property.

From \cite{Multiple}, let \(\nu\) be an ergodic self-joining of \((T_t^f)_{t\in \mathbb{R}}\), then \(\nu\) is either a product measure or a finite extension joining.
Suppose \(\nu\) is not the product measure,\index{\(\h\): The number of elements on the fiber} and it's \(\h\textit{-to-one}\). Then, from measurable selection\cite{measuretheory}, there are \(\h\) measurable functions \(\psi_1,...\psi_{\h}: X^f \rightarrow X^f\) such that the following set \index{\(\Omega_{\nu}\): A set of full \(\nu\) measure}
\begin{equation}
\label{The support of joining}
    \Omega_{\nu}\coloneqq\{((x,s),\psi_i(x,s)): (x,s)\in X^f, 1\leq i \leq \h\}
\end{equation}
is of full \(\nu\) measure.
Even though \[\{T_t^f(\psi_j(x,s)): 1\leq j \leq \h\}=\{\psi_j(T_t^f(x,s)): 1\leq j \leq \h\}\] for \(\mu\) \(a.e.\) \((x,s)\) and every \(t\in \mathbb{R}\), we are not guaranteed that \(\psi_i(T_t^f(x,s))=T_t^f\psi_i(x,s)\)  for each \(1\leq i \leq {\h}.\) We also know that \(\psi_i(x,s)\) and \(\psi_j(x,s)\) are not on the same orbit for \(i\neq j\) and \(\mu\) \(a.e.\) \((x,s).\) Otherwise, \(\nu\) is not singular w.r.t. \((Id\times T_r^f)_{*}\nu\) for some \(r\in \mathbb{R}.\) By ergodicity, \(\nu=\mu\otimes \mu\), which is a contradiction. Using the fact that \(\psi_i(x,s)\) and \(\psi_j(x,s)\) are not on the same orbit, we want to apply similar techniques in \Cref{Centralizer} to show that \(\nu\) must be 2-simple. Based on the distance between the orbits of \((x,s)\) and the orbits of the elements of the fiber of \((x,s)\), i.e. \(\psi_i(x,s)\), there are three cases.\\
\textit{Case 1:} There is some \(1\leq i \leq {\h}\), such that \((x,s)\) and \(\psi_i(x,s)\) is a good pair of order \(k\).\\
\textit{Case 2:} For all \(1\leq i \leq {\h}\), \((x,s)\) and \(\psi_i(x,s)\) are not good pairs of order \(k\). At some time \(t_0\), \(T_{t_0}^f(\psi_i(x,s))\) returns to a small neighborhood of \(\psi_i(x,s)\).\\
\textit{Case 3:} For all \(1\leq i \leq {\h}\), \((x,s)\) and \(\psi_i(x,s)\) are not good pairs of order \(k\). At some time \(t_0\), \(T_{t_0}^f(\psi_i(x,s))\) returns to a small neighborhood of \(\psi_j(x,s)\) with \(i\neq j\).

The proof of the second case is similar to the centralizer case. So, we discuss only the first and third cases.
We will prove two results in \Cref{Product Criteria} and \Cref{Simpleness Criteria} to rule out finite-to-one ergodic self-joinings which are not 2-simple. Then, we apply those results in our special flow setting in \Cref{Small shear general case} and conclude the proof of \Cref{Main}. First, we need a stronger uniform Birkhoff ergodic theorem result than \Cref{Classical uniform Birkhoff ergodic theorem} to prove the propositions.
Let \(\rho\) be an ergodic self-joining of an ergodic measure-preserving flow system \((X,(T_t)_{t\in\mathbb{R}},\mu,d)\). 
\noindent For every \(\e>0\) and every measurable set \(Y\subset X\) with \(\mu(\partial(Y))=0\),\index{\(V_{\e}\): \(\e\) neighborhood of a set}
let's denote \[V_{\e}(Y)=\{z\in X^f: d(z,Y)<\e\}.\]
\begin{DEF}
\index{\(X(\e)\): Almost continuous set}
\label{Almost continuous}
We say a flow \((T_t)_{t\in \mathbb{R}}\) is \textit{almost continuous} if for every \(\e>0\) there exists a set \(X(\e)\) with \(\mu(X(\e))>1-\e\) such that for every \(\e'>0\), \(\exists \delta'>0\) such that for every \(x\in X(\e)\),
\[d(T_t x, T_{t'} x)< \e', \quad \text{for}\; t,t'\in[-\delta',\delta'].\]
\end{DEF}
\begin{LEM}{\cite[Lemma 2.3]{Multiple}}
\label{Uniform Birkhoff ergodic theorem}
Let \(\{T_t\}_{t\in \mathbb{R}}\) be an ergodic flow acting on \((X,\mu,\mathcal{B},d)\) almost continuously and \(\rho\) be an ergodic self-joining of \((X,(T_t)_{t\in \mathbb{R}},\mu, d).\) Let \(K \subset \mathbb{R}\) be non-empty and compact and \(A, B\) be measurable sets with \(\mu(\partial A)=\mu(\partial B)=0\). Then, for every \(\eta, \theta ,\kappa>0\), there exist \(N=N(\eta,\theta,\kappa)\) and a set 
\index{\(\mathcal{Z}\): Set in the uniform Birkhoff ergodic theorem}
\(\mathcal{Z}=\mathcal{Z}(\eta,\theta,\kappa)\) with \(\rho(\mathcal{Z})>1-\theta\) such that for every \(M,L \geq N\) with \(L \geq \kappa M\) and \(p\in K\), we have
\[\left|\frac{1}{L} \int_{M}^{M+L} \mathds{1}_{A \times T_{-p}B}(T_tx,T_ty)dt-\rho(A\times T_{-p}B)\right|<\eta\]
and
\[\left|\frac{1}{L} \int_{M}^{M+L} \mathds{1}_{A \times T_{-p}B}(T_{-t}x,T_{-t}y)dt-\rho(A\times T_{-p}B)\right|<\eta\]
for every \((x,y)\in \mathcal{Z}.\)
\end{LEM}
\begin{RMK}
The above lemma works for every compact set \(K\). However, when we apply this lemma, we choose \(K=P\) as in \(\eqref{The compact set}.\) The original version of the lemma states a result for the discrete Birkhoff sum. However, it also works for the Birkhoff integral.
\end{RMK}
Suppose \(\rho\) is \(\h\)-to-one, so from measurable selection\cite{measuretheory}, we can find a set 
\begin{equation}
\label{Carrying set of abstract measure}
   \Omega_{\rho}\coloneqq\{(x,y_i(x)): x\in X, 1\leq i \leq \h\}
\end{equation}
such that \(\rho(\Omega_{\rho})=1.\)
For simplicity, we will write \(y_i(x)\) as \(y_i\) in the following. We also define 
\begin{equation}
\label{fiber of abstract measure}
   \Omega_{x}=\{y: (x,y)\in \Omega_{\rho}\}.
\end{equation}
\begin{DEF}\index{\(\e_1\)}
\label{Property I}
Let \((X,(T_t)_{t\in\mathbb{R}},\mu,d)\) be an ergodic measure-preserving flow system. Suppose for any ergodic self-joining \(\rho\) which is not the product measure, and let \(K\subset \mathbb{R}\setminus\{0\}\) be any fixed compact set. We can find \(\hat{\eta}>0\), \(\e_1>0\), \(R\coloneqq R(\hat{\eta})\) and measurable sets \(\hat{A_1},...,\hat{A_R}, \hat{B_1},...,\hat{B_R} \subset X\) such that for every \(0<\e<\e_1\) and \(p\in K\), \(\exists 1\leq j \leq R\) satisfying
\[|\rho(\hat{A_j} \times \hat{B_j})-\rho (V_{\e}(\hat{A_j})\times T_{-p}V_{\e}(\hat{B_j}))|\geq \hat{\eta}.\] Then we say \((X,(T_t)_{t\in\mathbb{R}},\mu,d)\) has \textit{variation property}. 
\end{DEF}
The following result is a standard Ratner shearing argument. Pick a base point \(x\) in a ``good'' set. Suppose we can find a map \(S\) that maps \(x\) to a nearby base point \(Sx\) in the good set, too. Now, the fiber points \(y\in\Omega_{x}\), \(y'\in\Omega_{x}\) are close to each other, which can be guaranteed by the definition of the good set. If we can find a certain time interval \([M,M+L]\) and two different values \(p_1,p_2\) depending on the distances between base points and fiber points, such that for most time \(t\in[M,M+L]\), \(T_{t}x\) is close to \(T_{t+p_1}Sx\) while \(T_{t}y\) is close to \(T_{t+p_2}y'.\) Then, we can compare \(T_t\times T_t\) orbits of \((x,y)\) and \((Sx,y')\) to show that \(\rho\) is the product measure.
\begin{PROP}
\label{Product Criteria}
Let \(K \subset \mathbb{R} \setminus\{0\}\)  be a given compact set. Let \(\rho\) be an ergodic self-joining of \((X,(T_t)_{t\in\mathbb{R}},\mu,d)\) satisfying the finite extension and variation properties. Assume that for every \(\e>0\) and \(N\in \mathbb{R}^+\),
there exist 
\begin{enumerate}
\item \(\kappa\coloneqq \kappa(\e)>0\)
\item Measurable set \(W\coloneqq W(\e)\subset X\) with  \(\mu(W)>1-\e\)
\item An injective measurable map \(S:W \to X\) where \(S\) preserves \(\mu\), that is \(\mu(A)=\mu(S^{-1}A)\) for every measurable set \(A\subset W\) 
\item A measurable set \(U\subset X\) with \(\mu(U)\geq 5\e\)
\end{enumerate}
so that for all \(x \in U\), \(\exists\) 
\(y \in \Omega_x\), \(y' \in \Omega_{Sx}\) and 
\begin{enumerate}[a)]
\item \(p_1\) or \(p_2\in K\) with \(p_1-p_2\in K\)
\item \(L,M>N\) with \(L\geq \kappa M\)
\end{enumerate}
such that at least one of the following is true:
\begin{equation}
\label{Forward relative shear}
  |\{t\in [M,M+L]: \max\{d(T_tx,T_{t+p_1}Sx),d(T_ty,T_{t+p_2}y')\}<\epsilon|>(1-\epsilon) L  
\end{equation}
and
\begin{equation}
\label{Backward relative shear}
 |\{-t\in [M,M+L]: \max\{d(T_tx,T_{t+p_1}Sx),d(T_ty,T_{t+p_2}y')\}<\epsilon|>(1-\epsilon) L.   
\end{equation}
Then, \(\rho=\mu \otimes \mu.\)
\end{PROP}
\begin{proof}
Suppose \(\rho\neq \mu \otimes \mu\), then we can assume \(\rho\) is a \(\h\)-to-one self-joining because \((X,(T_t)_{t\in\mathbb{R}},\mu,d)\) has the finite extension property.  Furthermore, for the given compact set \(K\), we can find \(\hat{\eta}\), \(R\), \(\e_1\) and sets \(\hat{A_j}\), \(\hat{B_j}\), \(V_{\e}(\hat{A_j})\) and \(V_{\e}(\hat{B_j})\)  for \(1\leq j \leq R\) since  \((X,(T_t)_{t\in\mathbb{R}},\mu,d)\) has the variation property. Take \(\e<\min\{\frac{\hat{\eta}}{1000},\e_1\}\). Apply \Cref{Uniform Birkhoff ergodic theorem} to sets \(\hat{A_j}\), \(\hat{B_j}\), \(V_{\e}(\hat{A_j})\), \(V_{\e}(\hat{B_j})\) and parameters \(\hat{\eta}/100, \frac{\e}{2\h}, \kappa\), we can obtain a set \(\mathcal{Z}_1\coloneqq\mathcal{Z}_1(\hat{\eta}/100,\frac{\e}{2\h},\kappa)\) with \(\rho(\mathcal{Z}_1)\geq 1-\frac{\e}{2\h}\) and a number \(N_1 \coloneqq N_1(\hat{\eta}/100, \frac{\e}{2\h}, \kappa)\) such that all of the following hold:
\[\left|\frac{1}{L} \int_{M}^{M+L} \mathds{1}_{A \times T_{-p}B}(T_tx,T_ty)dt-\rho(A\times T_{-p}B)\right|<\hat{\eta}/100,\]
\[\left|\frac{1}{L} \int_{M}^{M+L} \mathds{1}_{V_{\e}(A) \times T_{-p}V_{\e}(B)}(T_tx,T_ty)dt-\rho(V_{\e}(A) \times T_{-p}V_{\e}(B))\right|<\hat{\eta}/100,\]
\[\left|\frac{1}{L} \int_{M}^{M+L} \mathds{1}_{A \times T_{-p}B}(T_{-t}x,T_{-t}y)dt-\rho(A\times T_{-p}B)\right|<\hat{\eta}/100,\]
and
\[\left|\frac{1}{L} \int_{M}^{M+L} \mathds{1}_{V_{\e}(A) \times T_{-p}V_{\e}(B)}(T_{-t}x,T_{-t}y)dt-\rho(V_{\e}(A) \times T_{-p}V_{\e}(B))\right|<\hat{\eta}/100\]
for every \(p\in K\cup \{0\}\), \((x,y)\in \mathcal{Z}_1\) every pair of sets \(A\coloneqq \hat{A_j},B\coloneqq \hat{B_j}\) and \(M,L\geq N_1\) with \(L\geq \kappa M.\)
Let \[Y=\{x\in X: (x,y)\in \mathcal{Z}_1 \; \textit{for every}\; y\in \Omega_x\}.\]
Then, by the measure disintegration, we have
\[\frac{\e}{2\h}\geq\rho(\mathcal{Z}_1^c)=\int_X \rho_x(\mathcal{Z}_1^c) d\mu(x) \geq \frac{\mu(Y^c)}{\h}.\]
Therefore, \(\mu(Y)\geq 1-\e.\) Consider \(X'=(Y\cap W) \cap S^{-1}(Y\cap W)\), it follows that \(\mu(X'\cap U)>0\) because \(\mu(U)\geq 5\e\) by our assumption. Then pick \(x\in X'\cap U\) and \(y\in \Omega_x,y'\in \Omega_{Sx}\) as stated in the theorem. 
By \Cref{Property I}, for \(p=p_2-p_1\), there exist \(A\coloneqq \hat{A_j}\) and \(B\coloneqq \hat{B_j}\) for some \(1\leq j \leq R\) such that  
\begin{equation}
\label{Distance between measure of sets}
  |\rho(A \times B)-\rho (V_{\e}(A )\times T_{-p}V_{\e}(B))|\geq \hat{\eta}.  
\end{equation}
Without loss of generality, we assume that \eqref{Forward relative shear} is true.
On the one side, we have
\begin{align*}
        \rho(A\times B) &\leq \frac{1}{L} \int_{M}^{M+L} \mathds{1}_{A \times B}(T_tx,T_ty)dt +\frac{\hat{\eta}}{100} \\
        &\leq \frac{1}{L} \int_{M}^{M+L} \mathds{1}_{V_{\e}(A) \times B}(T_{t+p_1}Sx,T_ty)dt +\frac{\hat{\eta}}{100}+\frac{\e L}{L}\\
        &\leq \frac{1}{L} \int_{M}^{M+L} \mathds{1}_{V_{\e}(A) \times V_{\e}(B)}(T_{t+p_1}Sx,T_{t+p_2}y')dt +\frac{\hat{\eta}}{100}+\frac{2\e L}{L}\\
        &\leq \frac{1}{L} \int_{M}^{M+L} \mathds{1}_{V_{\e}(A) \times T_{-p}V_{\e}(B)}(T_{t}Sx,T_{t}y')dt +\frac{\hat{\eta}}{10}\\
        &\leq \rho(V_{\e}(A)\times T_{-p}V_{\e}(B))+\frac{\hat{\eta}}{5},\\
 \end{align*}
where the second and the third inequalities follow from \eqref{Forward relative shear}.
On the other side, using similar arguments, we can also obtain
\[ \rho(V_{\e}(A)\times T_{-p}V_{\e}(B)) \leq \rho(A\times B)+ \frac{\hat{\eta}}{5}.\]
This is a contradiction to \eqref{Distance between measure of sets}, hence \(\rho=\mu\otimes \mu.\)
\end{proof}
Observe that if  \(\rho\) is an ergodic self-joining of \((X,(T_t)_{t\in\mathbb{R}},\mu,d)\) which is not the product measure, then \(\rho\) and \((Id\times T_p)_{\ast}  \rho\) should be mutually singular when \(p\neq 0\). Based on this fact, we have the following definition.
\index{\(\e_2\)}
\begin{DEF}
\label{scale property}
Let \((X,(T_t)_{t\in\mathbb{R}},\mu,d)\) be an ergodic measure-preserving flow system. Suppose for any ergodic self-joining \(\rho\) which is not the product measure, any fixed \({\h}\in \mathbb{N}\) and any fixed compact set \(K\subset \mathbb{R}\setminus \{0\}\), there exist \(\e_2>0\), \(R'\in \mathbb{N}\) and measurable sets \(\hat{A_1},....,\hat{A_{R'}}, \hat{B_1},...,\hat{B_{R'}} \subset X\) such that for every \(0<\e<\e_2\) and \(p\in K\), \(\exists 1\leq j \leq R'\) satisfying
\[\rho(\hat{A_j} \times \hat{B_j})\geq 4 {\h}
 \rho (V_{\e}(\hat{A_j})\times T_{-p}V_{\e}(\hat{B_j})).\] Then we say  \((X,(T_t)_{t\in\mathbb{R}},\mu,d)\) has \textit{scale property}. 
\end{DEF}
\noindent For an ergodic self-joining \(\rho\) of \((X,(T_t)_{t\in\mathbb{R}},\mu,d)\) with the scale property which is not the product measure, denote 
\begin{equation}
\label{lower bound of R' sets}
  c(\rho,K,\h)\coloneqq \displaystyle\frac{1}{100}\min_{1\leq j \leq R'} \rho(\hat{A_j}\times \hat{B_{j}}) , 
\end{equation}
where \(R'\), \(\hat{A_j}\) and \(\hat{B_{j}}\) are given in \Cref{scale property}. 
The following result provides a way to show that a finite-to-one ergodic self-joining \(\rho\) is 2-simple. Recall the definitions in \eqref{Carrying set of abstract measure} and \eqref{fiber of abstract measure}. Suppose \(\rho\) is \(\h\)-to-one but not 2-simple, this means for a base point \(x\), there are at least two points \(y_1,y_2\in \Omega_x\) which are not on the same orbit. The main difference between the following proposition and other Ratner-type arguments in \cite{Multiple} and \cite{Disjoint} is that the distance between \(y_1\) and \(y_2\) is arbitrary. However, if one can show some relations between the base point \(x\) and some fiber point \(y\in \Omega_x\), this will allow us to obtain relations between \(y\) and other fiber points in \(\Omega
_x\) if \(\h\geq 2.\) We will show in \Cref{Small shear general case} that when the base point \(x\) and the fiber point \(y\in \Omega_x\) is a ``good'' pair, then the conditions in \Cref{Simpleness Criteria} will be satisfied. 
\begin{PROP}
\label{Simpleness Criteria}
Let \(\rho\) be an ergodic self-joining of \((X,(T_t)_{t\in\mathbb{R}},\mu,d)\) with scale property and \(\rho(\Omega_{\rho})=1\) as in \eqref{Carrying set of abstract measure}. Let \(K\subset \mathbb{R}\setminus\{0\}\) be a compact set, and \(\h\geq 2\). Then, there exist \(0<\e<\min\{\frac{c(\rho,K,\h)}{100}, \e_2\}\) and \(N>0\), such that there are \textbf{no} \(\kappa\coloneqq \kappa(\e)>0\), \(0<c_1,c_2\leq c(\rho,K,\h)\) and a measurable set \(B\subset X\) with \(\mu(B)\geq c_1\) satisfying for every \(x\in B\), there are \(p\in K\), and \(y_*\coloneqq y_i\in\Omega_x\) for some \(1\leq i \leq \h\) 
so that at least one of the following
\begin{equation}
\label{Forward ratio estimate}
   \min_{1\leq j \leq {\h}}d(T_t y_*,T_{t+p}y_j)\leq \e\ 
\end{equation}
and
\begin{equation}
\label{Backward ratio estimate}
   \min_{1\leq j \leq {\h}}d(T_{-t} y_*,T_{-t+p}y_j)\leq \e
\end{equation}
holds for every \(t\in \mathcal{I},\) where \(\mathcal{I}\subset[M,M+L]\) with \(L\geq N\), \(L\geq \kappa M\) and \(|\mathcal{I}|\geq (1-c_2)L .\)
\end{PROP}
\begin{proof}
Suppose, on the contrary, for every \(\e\) and \(N\), we can find \(\kappa(\e)\), \(c_1,c_2\) and \(B\) such that at least one of \eqref{Forward ratio estimate} and \eqref{Backward ratio estimate} is true. Take \(\eta\) to be a small number such that
\(c(\rho,K,\h) \geq {\h} \eta.\)
Take \(\theta\leq \frac{c_1}{200\h}\), applying \Cref{Uniform Birkhoff ergodic theorem} to \(R'\) pairs of sets \(\hat{A_j}\), \(\hat{B_j}\) in \Cref{scale property}, there exist \(N_2\coloneqq N_2
(\eta, \theta ,\kappa)\) and a set \(\mathcal{Z}_2\coloneqq \mathcal{Z}_2(\eta, \theta ,\kappa)\) with \(\rho(\mathcal{Z}_2)\geq 1-\theta\), such that for every \(L, M \geq N_2 \) and \(p\in K \cup{\{0\}}\), we have:
\begin{equation}
\label{Birkhoff Ratio Smaller}
  \left|\frac{1}{L}\int _{M}^{M+L} \mathds{1}_{\hat{A_j}\times T_{-p}\hat{B_j}}(T_t x , T_t y)d t- \rho(\hat{A_j}\times T_{-p}\hat{B_j})\right| \leq \eta  
\end{equation}
and
\begin{equation}
\label{Birkhoff Ratio Larger}
  \left|\frac{1}{L}\int _{M}^{M+L} \mathds{1}_{V_{\e}(\hat{A_j})\times T_{-p}V_{\e}(\hat{B_j})}(T_{t} x , T_{t} y)d t- \rho(V_{\e}(\hat{A_j})\times T_{-p}V_{\e}(\hat{B_j}))\right| \leq \eta  
\end{equation}
for every \((x,y)\in \mathcal{Z}_2\) and \(1\leq j \leq R'.\)
Let \[\mathcal{F}=\{x\in X: (x,y_j)\in \mathcal{Z}_2 \textit{ for every }1\leq j \leq {\h}\}.\]
From the measure disintegration,  
\[\theta\geq \rho(\mathcal{Z}_2^c)=\int_{\mathcal{F}^c} \rho_x(\mathcal{Z}_2^c) d\mu(x) \geq \frac{1}{{\h}}\mu(\mathcal{F}^c).\]
Therefore, \(\mu(\mathcal{F})\geq 1-{\h}\theta\geq 1-\frac{c_1}{2}\) and \(\mu(\mathcal{F} \cap B)>0\)  because \(\mu(B)\geq c_1.\)
Pick \(x\in \mathcal{F} \cap B.\) By our assumption,  there exist \(y_* \), \(M,L \geq N=N_2\) with \(L\geq \kappa M\), and \(p\in K\) such that for every \(t\in \mathcal{I}\)  with \(|\mathcal{I}|\geq(1-c_2)L\), there is \(y_j\) (\(j\) depending on \(t\)) such that 
\begin{equation}
\label{Distance between fiber elements}
  d(T_ty_*,T_{t+p}y_j) \leq \e.  
\end{equation}
By \Cref{scale property}, for such \(p\), there exist \(\hat{A_*}\) and \(\hat{B_*}\), such that 
\begin{equation}
\label{lower bound of measure}
  \rho(\hat{A_*}\times \hat{B_*})\geq 4{\h} \rho(V_{\e}(\hat{A_*})\times T_{-p}V_{\e}(\hat{B_*})).  
\end{equation}
Apply \eqref{Birkhoff Ratio Smaller} and \eqref{Birkhoff Ratio Larger} to  \(\hat{A_*}\) and \(\hat{B_*}\), we have
\begin{equation}
\label{upper bound of measure}
\rho(\hat{A_*}\times \hat{B_*}) \leq \frac{1}{L}\int _{M}^{M+L} \mathds{1}_{\hat{A_*}\times \hat{B_*}}(T_t x , T_t y_*)d t + \eta,
\end{equation}
and
\begin{equation}
\label{lower bound of larger measure sets}
 \rho(V_{\e}(\hat{A_*})\times V_{\e}(T_{-p}\hat{B_*})) \geq \frac{1}{L}\int _{M}^{M+L} \mathds{1}_{V_{\e}(\hat{A_*})\times T_{-p}V_{\e}(\hat{B_*})}(T_{t} x , T_{t} y_j)d t-\eta,  
\end{equation}
for every \(1\leq j \leq \h\) because \(x\in \mathcal{F}.\)
By finiteness of fiber elements and \eqref{Distance between fiber elements}, \(\exists y'\coloneqq y_{j_0}\) for some fixed \(1\leq j_0 \leq \h\) such that
\begin{equation}
\label{compare nearby orbits}
\begin{aligned}
    \frac{1}{L}&\int _{M}^{M+L} \mathds{1}_{V_{\e}(\hat{A_*})\times T_{-p}V_{\e}(\hat{B_*})}(T_{t} x , T_{t} y')d t  \\
   \geq  \frac{1}{{\h}L} &\int_{M}^{M+L}  \mathds{1}_{\hat{A_*}\times \hat{B_*}}(T_t x , T_t y_*)d t -\frac{c_2}{\h}.\\
\end{aligned}
\end{equation}
Combine \eqref{lower bound of measure}, \eqref{upper bound of measure}, \eqref{lower bound of larger measure sets} and \eqref{compare nearby orbits}, we can obtain 
\[\frac{1}{4{\h}}\rho(\hat{A_*}\times \hat{B_*}) \geq \rho(V_{\e}(\hat{A_*})\times T_{-p}V_{\e}(\hat{B_*})) \geq \frac{1}{{\h}} \rho(\hat{A_*}\times \hat{B_*}) -\frac{\eta}{\h}-\eta-\frac{c_2}{\h}.\]
Since \(c_2\leq c(\rho,K,\h)\leq \frac{\rho(\hat{A_*}\times \hat{B_*})}{100}\), this implies \(\rho(\hat{A_*}\times \hat{B_*})\leq 8 \h \eta\),
which is a contradiction to the fact that \(\rho(\hat{A_*}\times \hat{B_*})\geq 100 {\h}\eta.\)

\end{proof}

Now, we claim that the special flow system \((X^f,(T_t^f)_{t\in \mathbb{R}}, \mu^f)\) has variation and scale properties. 
\begin{LEM}
\label{Difference of sets}
For \(\alpha\in \mathcal{D}\) in \Cref{Definition of rotation angle} and \(f\) satisfies the condition in \Cref{Definition of roof}, the special flow system \((X^f,(T_t^f)_{t\in \mathbb{R}}, \mu^f)\) has variation property.
\end{LEM}
\begin{proof}
See proof of \cite[Theorem 4]{Multiple}.
\end{proof}
\begin{LEM}
\label{Choice of sets}
For \(\alpha\in \mathcal{D}\) in \Cref{Definition of rotation angle} and \(f\) satisfies the condition in \Cref{Definition of roof}, the special flow system \((X^f,(T_t^f)_{t\in \mathbb{R}}, \mu^f)\) has scale property.
\end{LEM}
\begin{proof}
Let \(\nu\) be an ergodic joining of \((X^f,(T_t^f)_{t\in \mathbb{R}}, \mu^f)\) that is not the product measure. Fixed a compact set \(K\subset \mathbb{R}\setminus \{0\}\) and an integer \(\h\geq 1.\) For every \(p\in K,\) we claim that there are measurable sets \(A_p, B_p\subset X^f\) with \(\mu^f(\partial A_p)=\mu^f(\partial B_p)=0\) satisfying
\(\nu(A_p \times B_p)\geq 16 {\h}
 \nu (A_p \times T_{-p}^fB_p).\)
Suppose this is not true, by measure disintegration in \cite{measuretheory}, for any measurable sets \(A\) and \(B\), we have
\(\nu(A\times B)= \int_A \nu_{y} (B) d\mu^f(y).\)
Now, consider the following family of sets:
\[\mathcal{Q}=\{C(x,r)\subset X^f: \text{\(C(x,r)\) is a cube with center \(x\in \mathbb{Q}^2\) and radius \(r\in \mathbb{Q}\)}\}.\]
Then for every \(B\in \mathcal{Q},\) and \(\mu^f\)-almost every \( y\in X^f,\) \(\nu_y(B)<16\h\nu_y(T_{-p}^fB).\)  However, when \(\h=1\), one can choose a cube \(B\in \mathcal{Q}\) with small diameter containing the atom so that \(\nu_y(B)=1\) while \(\nu_y(T_{-p}^fB)=0,\) which is a contradiction. When \(\h> 1\), note that for \(a.e.\) \(y\), and any atom \(z\) in fiber of \(y\), we have \(\nu_y(\{T_t^fz: t\neq 0\})=0.\) Otherwise, \(\nu_y\) will be an infinite measure due to ergodicity of \(\nu\) and aperiodicity of \(T_t^f.\) Since we only have a finite number of atoms and \(p\) is in the given compact set \(K\), we can also choose a cube \(B\in \mathcal{Q}\) with small diameter containing an atom so that \(\nu_y(B)=\frac{1}{\h}\) but \(\nu_y(T_{-p}^fB)=0\). Therefore, we prove the claim. Choose \(0<\e_p< \nu(A_p\times T_{-p}^fB_p).\) Since \(\mu^f(\partial A_p)=\mu^f(\partial B_p)=0\), there exists \(0<\e'<\e_p\) so that \(\nu(V_{2\e'}A_p\times T_{-p}^f V_{2\e'}B_p)\leq 2 \nu(A_p\times T_{-p}^f B_p)\). Recall \Cref{Almost continuous}, note that \(T_t^f\) is an almost continuous flow, which follows from its definition. 
For such \(\e',\) there exists \(0<\delta_p<\e'/2\) such that \(d(T_t^fx,T_{t'}^fx)<\e'\) for every \(x\in X(\e')\) and \(t,t'\in[-\delta_p,\delta_p].\) Moreover,  \(\mu^f(X(\e'))\geq 1-\e_p.\)
For any \(p'\in[p-\delta_p,p+\delta_p]\), we can then obtain 
\begin{align*}
   &\nu(V_{\delta_p} A_p \times T_{-p'}^f V_{\delta_p} B_p)\\
\leq &\nu\left(V_{\delta_p} A_p \times T_{-p}^f \bigcup_{t\in[-\delta_p,\delta_p]} T_t^f V_{\delta_p}B_p\right)\\
\leq & \nu\left(V_{\delta_p} A_p \times T_{-p}^f \left(\bigcup_{t\in[-\delta_p,\delta_p]} T_t^f V_{\delta_p}B_p\cap X(\e')\right)\right)\\
+& \nu\left(V_{\delta_p} A_p \times T_{-p}^f \left(\bigcup_{t\in[-\delta_p,\delta_p]} T_t^f V_{\delta_p}B_p\cap X(\e')^c\right)\right)
\\\leq & \nu\left(V_{\delta_p} A_p \times T_{-p}^f V_{2\delta_p}B_p\right)+\mu^f(X(\e')^c)\\
\leq & 2\nu(A_p\times T_{-p}^f B_p)+\e_p\\
\leq & 4\nu(A_p\times T_{-p}^f B_p).
\end{align*}
This implies \(\nu(A_p\times B_p)\geq 4\h \nu(V_{\delta_p} A_p \times T_{-p'}^f V_{\delta_p} B_p).\)
Consider an open covering of \(K\)
\[\bigcup_{p\in K} (p-\frac{\delta_p}{3},p+\frac{\delta_p}{3}).\]
Since \(K\) is a compact set, there is a finite open covering,
\[\bigcup_{i=1}^{R'} (p_i-\frac{\delta_{p_i}}{3},p_i+\frac{\delta_{p_i}}{3})\]
where \(p_i \in K.\)
Pick \(\e_2=\frac{1}{2}\displaystyle \min_{1\leq i\leq R'} \{\delta_{p_i}\}.\)
Hence, there exist sets \(\hat{A_1},....,\hat{A_{R'}}\), \(\hat{B_1},...,\hat{B_{R'}} \) such that for every \(p\in K\) and \(\e<\e_2\), \(\exists 1\leq j \leq R'\) satisfying
\[\nu(\hat{A_j} \times \hat{B_j})\geq 4 {\h}
 \nu (V_{\e}(\hat{A_j} )\times T_{-p}^fV_{\e}(\hat{B_j})) 
.\] 
\end{proof}
The following lemma, together with \Cref{Relative shear special flow}, will be used to show that the self-joining \(\nu\) of the special flow system \((X^f,(T_t^f)_{t\in\mathbb{R}},\mu^f)\) with \(\nu(\Omega_{\nu})=1\) is 2-simple. Then we can claim that the special flow system has the minimal self-joining property by \Cref{Theorem of centralizer}. The proof of the following lemma is similar to the proof in \Cref{Small shear flow}. This completes the proof of \Cref{Main} of the first case using \Cref{Simpleness Criteria}. Let's denote \index{\(\Omega(x,s)\):Fiber set of base point \((x,s)\)}
\begin{equation}
\label{Definition of fiber of (x,s)}
  \Omega(x,s)=\{(y^j,s^j): ((x,s),(y^j,s^j))\in \Omega_{\nu}, 1\leq j \leq {\h}\}.  
\end{equation}
Recall that \(G_k=\{((x,s_1),(y,s_2))\in X^f\times X^f: (x,y) \textit{\;is a good pair of order k}\}\) where the definition of good pair of order \(k\) is in \Cref{Definition of good pair}. 
Since \((X^f,(T_t^f)_{t\in\mathbb{R}},\mu^f)\) has scale property, and \(\nu\) is not the product measure, choose \(c(\nu,P,\h)\) as in \eqref{lower bound of R' sets}
and \(\e_2\) in \Cref{scale property} for \(\nu\) where \(P\) is the compact set in \eqref{The compact set}.
\begin{LEM}
\label{Small shear general case}
Suppose \(\alpha\in \mathcal{D}\) in \Cref{Definition of rotation angle} and \(f\) satisfies the condition in \Cref{Definition of roof}, let \(\nu\) be an ergodic self-joining of \((X^f,(T_t^f)_{t\in \mathbb{R}}, \mu^f)\) and \(\nu(\Omega_{\nu}
)=1\) with \(\Omega_{\nu}\) in \eqref{The support of joining}. Given \(0<\e<\min\{\frac{c(\nu,P,\h)}{100}, \e_2\}\), suppose that \(\nu( G_k)>\e\) for large enough \(n_k\), then \(\nu\) is 2-simple.  
\end{LEM}
\begin{proof}
By Lusin's theorem, given any  \(\e>0\), there is a subset \(\tilde{W}\subset X^f\) with \(\mu^f(\tilde{W})>1-\frac{\e}{200 \h}\), such that \(\psi_1,...,\psi_{\h}\) are uniformly continuous on \(\tilde{W}.\) Thus,  there exists \(\tau'\), such that if \((x,s_1), (y,s_2) \in \tilde{W}\) and \(\|x-y\|+|s_1-s_2|\leq \tau'\), then 
\[d(\psi_j(x,s_1),\psi_j(y,s_2)) \leq \delta(\tau')\]
for every \(1\leq j \leq {\h}.\)
There exists \(\tau_{\h}\), such that if \(\tau'<\tau_{\h}\), then \(\delta(\tau')\leq\delta\) where \(\delta\) is as in \eqref{Value of delta}. Denote \(\tilde{X}\coloneqq\{(x,s)\in X_{\frac{\e}{5\h}}: (y,s')\in X_{\frac{\e}{5\h}} \text{ for all } (y,s')\in\Omega(x,s)\}\) with \(X_{\frac{\e}{5\h}}\) as in \eqref{Set X_{e}}.
By measure disintegration of \(\nu\),
\[\frac{\e}{100\h}\geq\nu((X_{\frac{\e}{5\h}}\times X_{\frac{\e}{5\h}})^c)\geq \int_{\tilde{X}^c} \nu_{x}((X_{\frac{\e}{5\h}}\times X_{\frac{\e}{5\h}})^c) d\mu^f(x)\geq \mu^f(\tilde{X}^c) \frac{1}{\h}.\]
Define 
\begin{equation}
\label{Definition of Y}
 \mathcal{Y}\coloneqq \tilde{X}\cap \tilde{W},
\end{equation}
so we have \(\mu^f(\mathcal{Y})\geq 1-\frac{\e}{50}.\)
By Birkhoff ergodic theorem and \(\mu^f\) ergodicity of \(T^f,\) there exist a set \(\mathcal{W}(\e)\) with \(\mu^f(\mathcal{W}(\e))\geq 1- \frac{\e}{100}\) and \(T_{\h}(\e)>0,\) such that for every \(T'\geq T_{\h}(\e)\) and every \((x,s)\in \mathcal{W}(\e)\), we have
\[\left|\frac{1}{T'}\int_0^{T'} \mathds{1}_{\mathcal{Y}}\left(T_t^f(x,s)\right) d t\right|\geq 1-\frac{\e}{25 }.\]
If \(\nu\) is not 2-simple, this means \(\h\geq 2.\) 
Recall that \(Z(\e)\) is the set satisfying the SWR property in \eqref{Switchable Ratner set} and \( O(\e)\) is the set in \eqref{Uniform set O}.
Denote \[\tilde{B}\coloneqq G_k\cap \left(\left(Z(\frac{\e}{100\h})\cap O(\frac{\e}{100\h})\right)^f \times E_{n_k}^f\right)\]
and 
\(\pi_1:X^f\times X^f \rightarrow X^f\) to be map such that
\[\pi_1((x,s_1),(y,s_2))=(x,s_1).\]
Apply \Cref{Simpleness Criteria} to \(\pi_1(\tilde{B})\cap \mathcal{W}(\e
)\), \(c_1=\frac{\e}{2}, c_2=\e\) and \(K=P\) in \eqref{The compact set}. For any \(N>0,\) picking \(n_k\) large enough so that \(\frac{q_{n_k}}{a}\geq N\) and \(q_{n_k} \geq \max\{T_{\h}(\e)/2a,\frac{1}{\tau_0^3},\frac{1}{\delta^3},10^8\}\) where \(\tau_0\) is in \eqref{Value of delta} and \(a\) is the lower bound of the roof function. 
Furthermore, \[n_k\geq \max \{\N_0(\frac{\e}{100\h}),\N_1(\frac{\e}{100\h}),\N_2(\frac{\e}{100\h},\frac{\e}{1000\h}),\N_3(\frac{\e}{100\h}),\N_4(\frac{\e}{5\h})\}.\] 
By \Cref{Simpleness Criteria}, it suffices to show we can find a \(\kappa\) so that for every \((x,s)\in \pi_1(\tilde{B})\cap \mathcal{W}(\e)\), there are \(p\in P\), and \((y^{*},s^{*})\in \Omega(x,s)\) such that at least one of the following
\begin{equation}
\label{Forward y_*}
 \min_{1\leq j \leq {\h}}d(T_t^f(y^{*},s^{*}),T_{t+p}^f(y^i,s^i))\leq \e
\end{equation}
and
\begin{equation}
\label{Backward y_*}
  \min_{1\leq j \leq {\h}}d(T_{-t}^f(y^{*},s^{*}),T_{-t+p}^f(y^i,s^i))\leq \e
\end{equation}
holds for every \(t\in \mathcal{I},\) where \(\mathcal{I}\subset[M,M+L]\) with \(L\geq N\), \(L\geq \kappa M\) and \(|\mathcal{I}|\geq (1-\e)L .\)
Take \(((x,s),(y^{*},s^{*}))\in\tilde{B}\), by \Cref{Small Shearing}, there exist \(\ell_0\), \(m\geq n_k\) and \(p\in P\) depending on \(x,y^{*}\) such that 
\[|f^{(\ell_0 q_m)}(x)-f^{(\ell_0 q_m)}(y^{*})+p|\leq \frac{\e}{1000\h}.\]
If \(\ell_0>0,\) we prove \eqref{Forward y_*} is true. If \(\ell_0<0\), we prove \eqref{Backward y_*} is true. We present the proof for the case when \(\ell_0>0\). The case of \(\ell_0<0\) is similar. Consider the orbits \(\{T_t^f(x,s): t\in[0, f^{(2\ell_0 q_m)}(x)]\} \).
Since \((x,s)\in \mathcal{W}(\e),\) the set \(\Lambda=\{t\leq f^{(2\ell_0 q_m)}(x) : T_t^f(x,s) \in \mathcal{Y}\}\) has density at least \(1-\frac{\e}{10}\) where \(\mathcal{Y}\) is as in \eqref{Definition of Y}.
For each \(t\leq f^{([0.99\ell_0 q_m])}(x)\), let \(x_t=x+n(x,s,t)\alpha\) and \(y_t^{*}=y^{*}+n(y^{*},s^{*},t)\alpha.\) Also recall \(s_t\) and \(s_t^*\) as in \eqref{Coordinates of special flow}. From \Cref{Small shearing preserving}, for every \(0\leq n(x,s,t) \leq \ell_0 q_m\) and corresponding \(n(y^{*},s^{*},t)\), we have:
\begin{equation}
\label{Birkhoff sum difference between x_t and y_t*}
|f^{(\ell_0 q_m)}(x_t)-f^{(\ell_0 q_m)}(y_t^{*})+p|\leq \frac{\e}{1000\h}. 
\end{equation}
For each \(t \in \Lambda,\) and \(t\leq f^{([0.99\ell_0 q_m])}(x)\), consider \((\tilde{x}_t,\tilde{s}_{t})=T_{f^{(\ell_0 q_m)}(x_t)}^f(x_t,s_t).\) Similarly, there is a set \(\Lambda'\subset [0, f^{([0.99\ell_0 q_m])}(x)]\) with density larger than \(1-\e/2\) such that for every \(t\in \Lambda',\) both \((x_t,s_t)\) and \((\tilde{x}_t,\tilde{s}_t)\) are in \(\mathcal{Y}.\) Therefore, 
\begin{equation}
\label{Continuity of fiber elements}
d(\psi_j(x_t,s_t),\psi_j(\tilde{x}_t,\tilde{s}_t))\leq \delta  
\end{equation}
for \(t\in \Lambda'\) and every \(1\leq j \leq \h.\) By \eqref{Birkhoff sum difference between x_t and y_t*}, triangle inequality and our assumptions on \(n_k\), we have
\begin{equation}
\begin{aligned}
\label{triangle inequality for (y_t*,s_t*)}
    &d((y_t^{*},s_t^{*}), T_{f^{(\ell_0 q_m)}(x_t)+p}^f(y_t^{*},s_t^{*}))\\
\leq & d((y_t^{*},s_t^{*}), T_{f^{(\ell_0 q_m)}(y_t^{*})}^f(y_t^{*},s_t^{*}))+d(T_{f^{(\ell_0 q_m)}(x_t)+p}^f(y_t^{*},s_t^{*}),T_{f^{(\ell_0 q_m)}(y_t^{*})}^f(y_t^{*},s_t^{*}))\\
\leq & \delta+\frac{\e}{1000\h}.
\end{aligned} 
\end{equation}
However, for every \(t\in \Lambda'\), there exists \(1\leq i'\coloneqq i'(t) \leq {\h}\) such that 
\begin{equation}
\label{i' th fiber}
    \psi_{i'}(\tilde{x}_t,\tilde{s}_t)=T_{f^{(\ell_0q_m)}(x_t)}^f(y_t^{*},s_t^{*}).
\end{equation}
Combine \eqref{Continuity of nearby points within distance P}, \eqref{Continuity of fiber elements}, \eqref{triangle inequality for (y_t*,s_t*)} and \eqref{i' th fiber}, for \(t\in \Lambda''\subset [0, f^{([0.99\ell_0 q_m])}(x)]\) with density larger than \(1-\e,\) there exists \(1\leq i\coloneqq i(t) \leq {\h}\) such that \(T_t^f(y^i,s^i)=\psi_{i'}(x_t,s_t)\) and
\begin{equation}
\label{figure}
\begin{aligned}
&d(T_t^f(y^{*},s^{*}),T_{t+p}^f(y^i,s^i))\\
=&d((y_t^{*},s_t^{*}),T_{p}^f\psi_{i'}(x_t,s_t))\\
\leq & d((y_t^{*},s_t^{*}),T_{p}^f\psi_{i'}(x_t,s_t)) + d(T_{p}^f\psi_{i'}(x_t,s_t), T_{p}^f\psi_{i'}(\tilde{x}_t,\tilde{s}_t)) \\
\leq & \delta+\frac{\e}{1000\h}+ \frac{\e}{5000\h}\leq \e.\\
\end{aligned}  
\end{equation}
 Therefore, we choose \(M=0,\) \(L=f^{([0.99\ell_0 q_m])}(x)\) and \(\kappa=1\). Applying \Cref{Simpleness Criteria}, we obtain a contradiction if \(\h\geq 2\).
\end{proof}
\begin{figure}[H]
\begin{centering}
\begin{tikzpicture}[scale=1.4]
  \draw[->] (-0.5,0) -- (6,0) node[right] {};
  \draw[->] (0,-0.5) -- (0,4) node[above] {};

  \node at (-0.2,-0.2) {O};

  \draw[thick, domain=0.8:5.2, smooth, variable=\x] plot ({\x}, {0.4*(\x - 3)^2 + 1});

  \filldraw (1.3,0.5) circle (1pt) node[below] {$(\tilde{x}_t, \tilde{s}_t)$};
  \filldraw (2.2,0.5) circle (1pt) node[below] {$(x_t, s_t)$};

  \filldraw (5,1.96) circle (1pt) node[right] {$(y_t^*, s_t^*)$};

  \draw[decorate,decoration={brace,amplitude=7pt},xshift=0pt]
        (5,1.96) -- (5,0.4) node[midway,xshift=8pt,right] {$p$};

  \draw[decorate,decoration={brace,amplitude=7pt},xshift=0pt]
        (4.3,0.5) -- (4.9,0.5) node[midway,yshift=4pt, above] {$\delta$};
  \filldraw (4.3,0.5) circle (1pt) node[left] {$\psi_{i'}(x_t, s_t)$};
\filldraw (4.9,0.4) circle (1pt) node[below] {$\psi_{i'}(\tilde{x}_t,\tilde{s}_t)$};
\end{tikzpicture}    
\begin{center}
\caption{A figure about \eqref{figure}.}
\end{center}
\end{centering}
\end{figure}
From the above lemma, it suffices to consider \textit{Case 3} for ``most'' points in \(\Omega_{\nu}\). To state the idea of proof, we assume \({\h}=2\) first. Pick \((x,s)\) in a ``good'' set, then we map \((x,s)\) to \((x',s)=(x+q_{n}\alpha,s)\) where \(n\) is a large enough number. Consider the corresponding fiber points \((y,\tilde{s})\in \Omega(x,s)\) and \((y',s')\in \Omega(x',s)\). By running an argument in \Cref{Product Criteria}, we can obtain a distance estimate for \(y\) and \(y'.\) We will show if \(\|y-y'\|\) is not ``close'' to \(\|q_{n}\alpha\|\), then the conditions in \Cref{Product Criteria} are satisfied. Thus, the measure \(\nu\) is the product measure, which is a contradiction. To be more precise, for every \(\e\) and \(N\), we are able to find a long time interval \([M,M+L]\) with \(L \geq N\) and \(p_1,p_2 \in P\) with \(p_1\neq p_2\) such that for most times \(t\in [M,M+L]\) we have 
\[d(T_t^f (x,s), T_{t+p_1}^f(x',s)) < \e \text{  and  } d(T_t^f (y,\tilde{s}), T_{t+p_2}^f (y',s')) < \e.\]
\noindent Therefore, it suffices to consider the case where \(\|y-y'\|\) and \(\|x-x'\|\) are ``close'' enough. This is also the main difference between \(\h\geq 2\) and \(\h=1\). In the proof of \Cref{Theorem of centralizer}, the distance \(\|y-y'\|\) is much larger than \(\|x-x'\|\) because of the large shearing estimate \Cref{Large Shearing}.
The issue is that \(y\) and \(y'\) may not be on the same orbit if \(\h\geq 2.\)
However, if we apply the above methods again to the pair \(((x',s),(y',s'))\) in the good set. We can obtain a new pair \(((x'',s),(y'',s''))\) and \(\|y'-y''\|\) is also ``close'' to \(\|x'-x''\|=\|x-x'\|\). Since we assume \(\h=2,\) this implies \(y\) and \(y''\) must be in the same orbit, so we can use the large shearing result to obtain a contradiction because \(\|y-y''\|\) should be much larger than \(\|x-x''\|.\) For general case \(\h\), the idea is similar. However, it will complicate the argument and notations in the proof of \Cref{Main}.
\begin{LEM}
\label{Relative shear special flow}
For every \(\e<\min\{\e_1,\e_2, \frac{c(\nu,P,\h)}{100},\frac{\hat{\eta}}{1000}\}\), \(N>0\), there exist \(\kappa\coloneqq\kappa(\e)\) and \(\delta'\coloneqq \delta(\e,N)<\frac{\e^2}{1000}\) such that if for \(((x,s),(y,\s))\) and \(((x',s), (y',s')) \in \Omega_{\nu}\) with  \(x,y\in E_{\ell} \cap Z(\frac{\e}{10\h}) \cap O(\frac{\e}{10\h})\) and \(x'=x+q_{\ell}\alpha\) where \(\ell=n_k\) or \(n_k+1\) for some large enough \(n_k\) satisfying \[\frac{4}{3} \|x-x'\| \leq \|y-y'\|\leq \delta'\textit{\; or \; } \|y-y'\|\leq \frac{2}{3}\|x-x'\|\leq \delta', \] and \(|s'-\s|\leq \delta'\). Then, we can find \(M, L\) and \(p_1,p_2\) with \(p_1-p_2\in P\) and \(L\geq \kappa M\) such that at least one of the following cases is true:\\
Forward case:
\begin{equation}
\label{Forward case induction}
      d(T_t^f (x,s), T_{t+p_1}^f (x',s)) \leq \e \textit{\; and \;} 
  d(T_t^f (y,\tilde{s}), T_{t+p_2}^f (y',s')) \leq \e
\end{equation}
Backward case:
\begin{equation}
\label{Backward case induction}
    d(T_{-t}^f (x,s), T_{-t+p_1}^f (x',s)) \leq \e \textit{\; and \;} 
  d(T_{-t}^f (y,\tilde{s}), T_{-t+p_2}^f (y',s')) \leq \e
\end{equation}
for \(t\in U\subset [M,M+L]\) with \(|U|\geq (1-\e) L.\)
\end{LEM}
\begin{proof}
For \(\e'\leq \frac{\e}{1000}\) and \(N'\geq 2N\), we will show a Birkhoff sum difference argument for the forward orbit case; the proof of the backward orbit case is similar. Then, from the Birkhoff sum difference argument, we can conclude \eqref{Forward case induction}. Assume that \(\|y-y'\| \leq \frac{2}{3}\|x-x'\|\) and \(\ell=n_k\) first. By picking \(\ell\) large enough, since \(x\in E_{\ell}\cap Z(\frac{\e}{10\h})\cap O(\frac{\e}{10\h})\), we claim that we can find \(\kappa'\), \(M'\), \(L'\), \(p_1,p_2\) with \(p_1-p_2\in P\) such that
\begin{equation}
\label{p_1 estimate}
    \left|f^{(n)}(x')-f^{(n)}(x)-p_1\right|\leq \e',
\end{equation}
\begin{equation}
\label{p_2 estimate}
    \left|f^{(n)}(y')-f^{(n)}(y)-p_2\right|\leq \e',
\end{equation}
and
\begin{equation}
\label{relative estimate}
    \left|\left(f^{(m)}(y')-f^{(m)}(y)\right)-\left(f^{(n)}(y')-f^{(n)}(y)\right)\right|\leq \e'
\end{equation}
for every \(n\in[M',M'+L']\) with \(L'\geq \kappa' M'\) and \(|n-m|\leq M'^{1/2}.\) Since \(x,y\in E_{\ell},\) for \(m_0\in \{1,2,...,\frac{q_{\ell+1}}{16q_{\ell}}\}\), there exist \(x_{\ell}\) and \(y_{\ell}\) such that
\[|f^{(m_0q_{\ell})}(x)-f^{(m_0q_{\ell})}(x')|=|f'^{(m_0q_{\ell})}(x_{\ell})|\|x-x'\|\]
and
\[|f^{(m_0q_{\ell})}(y)-f^{(m_0q_{\ell})}(y')|=|f'^{(m_0q_{\ell})}(y_{\ell})|\|y-y'\|.\]
Apply the same estimate shown in \Cref{Small Shearing}, we know that 
\[m_0Hq_{\ell} \log(q_{\ell}) \leq |f'^{(m_0q_{\ell})}(x_{\ell})|\leq m_0(H+1)q_{\ell}\log(q_{\ell})\]
and
\[m_0Hq_{\ell} \log(q_{\ell}) \leq |f'^{(m_0q_{\ell})}(y_{\ell})|\leq m_0(H+1)q_{\ell}\log(q_{\ell}). \]
One can pick \(m_*\leq \frac{q_{\ell+1}}{16q_{\ell}}\) such that 
\[\frac{H}{32}\leq|f^{(m_*q_{\ell})}(x)-f^{(m_*q_{\ell})}(x')|\leq \frac{H+1}{16}.\]
Since \(\|y-y'\|\leq \frac{2}{3}\|x-x'\|\),
\[|f^{(m_*q_{\ell})}(y)-f^{(m_*q_{\ell})}(y')|\leq \frac{3}{4}|f^{(m_*q_{\ell})}(x)-f^{(m_*q_{\ell})}(x')|.\]
Hence, we can pick \(p_1\in P\) and \(p_2\) such that \(f^{(m_*q_{\ell})}(x)-f^{(m_*q_{\ell})}(x')=p
_1,\) \(f^{(m_*q_{\ell})}(y)-f^{(m_*q_{\ell})}(y')=p_2\) and \(p_1-p_2\in P.\) By taking \(M'=m_*q_{\ell},\) \eqref{p_1 estimate} and \eqref{p_2 estimate} follow from cocycle identity as in the proof of \cite[Proposition 4.5]{Multiple} part \(b\). \eqref{relative estimate} follows from the fact that \(x,y\in O(\frac{\e}{10\h}).\)
If \(\frac{4}{3} \|x-x'\| \leq \|y-y'\|\), since \(x\in E_{\ell}\) and \(y\in Z(\frac{\e}{10\h})\), the proof is similar as before.
If \(\ell=n_k+1\) and \(q_{\ell+1}\geq q_{\ell}\log(q_{\ell}),\) we can repeat the proof as before.
If \(\ell=n_k+1\) and \(q_{\ell+1}\leq q_{\ell}\log(q_{\ell})\), we can replace \(\ell\) with \(n_k\) instead; the rest of the proof is similar to before.
Recall \(n(x,s,t)\) and \(s_t\) in \eqref{Coordinates of special flow}.
Since \(|n(x,s,t+1)-n(x,s,t)|\leq \frac{1}{a},\) we can find \(M,L\) such that
\[n(x,s,M)\in [M'+1,M'+1+\frac{1}{a}]\]
and
\[n(x,s,M+L)\in [M'+L'-1,M'+L'-1-\frac{1}{a}].\]
Take \(\kappa(\e)=\frac{\kappa'}{100},\) we can obtain that
\(L\geq \kappa M\) using \Cref{Sum Ostrowski}.  \\
Let \[U(x,s)=\{t\in[M,M+L]: \frac{\e}{100}\leq s_t \leq f(x+n(x,s,t)\alpha)-\frac{\e}{100} \},\]
then we know \(|U(x,s)|\geq (1-\frac{\e}{3})L.\)
Take \(U=U(x,s)\cap U(y,\s)\), then \(|U|\geq(1-\e)L.\)
For \(t\in U\), by definition,
\[ f^{(n(x,s,t))}(x)+\frac{\e}{100}\leq s+t\leq f^{(n(x,s,t)+1)}(x)-\frac{\e}{100}.\]
But from \eqref{p_1 estimate}, on the one side
\begin{equation}
\begin{aligned}
      t+s+p_1&\leq t+s+\e'+f^{(n(x,s,t))}(x')-f^{(n(x,s,t))}(x)\\
      &\leq f(x+n(x,s,t)\alpha)-\frac{\e}{100}+f^{(n(x,s,t))}(x')+\e'\\
      &\leq f^{(n(x,s,t)+1)}(x')
\end{aligned}
\end{equation}
because \(x'=x+q_{\ell}\alpha\) and \(x\in E_{\ell}\cap Z(\frac{\e}{10\h}).\)
On the other side, we can obtain
\begin{equation}
\begin{aligned}
      t+s+p_1&\geq t+s-\e'+f^{(n(x,s,t))}(x')-f^{(n(x,s,t))}(x)\\
      &\geq \frac{\e}{100}+f^{(n(x,s,t))}(x')-\e'\\
      &\geq f^{(n(x,s,t))}(x').
\end{aligned}
\end{equation}
Hence, \[T_{t+p_1}^f(x',s)=(x'+n(x,s,t)\alpha,s+t+p_1-f^{(n(x,s,t))}(x'))\] and \[d(T_{t}^f(x,s),T_{t+p_1}^f(x',s))\leq \e.\]\\
By \eqref{relative estimate}, we can assume without loss of generality that \(n(y,\s,t)\in [M',M'+L']. \)
Then by \eqref{p_2 estimate}, we have:
\begin{equation}
\begin{aligned}
   t+s'+p_2&\leq t+\s+\e'+f^{(n(y,\s,t))}(y')-f^{(n(y,\s,t))}(y)+s'-\s\\
      &\leq f(y+n(y,\s,t)\alpha)-\frac{\e}{100}+f^{(n(y,\s,t))}(y')+\e'+\frac{\e^2}{1000}\\
      &\leq f^{(n(y,\s,t)+1)}(y'),
\end{aligned}
\end{equation}
and
\begin{equation}
\begin{aligned}
   t+s'+p_2&\geq t+\s-\e'+f^{(n(y,\s,t))}(y')-f^{(n(y,\s,t))}(y)+s'-\s\\
      &\geq \frac{\e}{100}+f^{(n(y,\s,t))}(y')+\e'-\frac{\e^2}{1000}\\
      &\geq f^{(n(y,\s,t))}(y').
\end{aligned}
\end{equation}
Therefore, \[T_{t+p_2}^f(y',s')=(y'+n(y,\s,t)\alpha,s'+t+p_2-f^{(n(y,\s,t))}(y'))\] and \[d(T_{t}^f(y,\s),T_{t+p_2}^f(y',s'))\leq \e.\] We complete the proof of this lemma.
\end{proof}
\begin{RMK}
When we want to prove the divergence of nearby orbits of special flow, we first prove the divergence of Birkhoff sum differences. The order of quantifiers here is important, and they are different for flow and Birkhoff sum cases. We first obtain \(\kappa'(\e')\), \(M'\) and \(L'\) from the Birkhoff sum estimate, then we obtain \(\kappa(\e)\), \(M\) and \(L\) for the special flow.
\end{RMK}
\begin{proof}[Proof of \Cref{Main}]
Suppose \(\alpha\in \mathcal{D}\) in \Cref{Definition of rotation angle} and \(f\) satisfies the condition in \Cref{Definition of roof}, let \(\nu\) be an ergodic self-joining of \((X^f,(T_t^f)_{t\in \mathbb{R}}, \mu^f)\) and \(\nu(\Omega_{\nu}
)=1\) with \(\Omega_{\nu}\) in \eqref{The support of joining}.
By \Cref{Difference of sets} and \Cref{Choice of sets}, the special flow system has variation and scale properties. Therefore, we can find \(\e_1, \hat{\eta}, \e_2\) as in \Cref{Property I} and in \Cref{scale property}. Recall the definition of \( c(\nu,P,\h)\) in \eqref{lower bound of R' sets}.
Fixed \(\e<\min\{\e_1,\e_2, \frac{c(\nu,P,\h)}{100},\frac{\hat{\eta}}{1000}\}\) and an arbitrary number \(N.\) From \(\e\) and \(N\), we can get \(\delta(\frac{\e}{50\h},N)\) and \(\kappa(\frac{\e}{50\h})\) by applying \Cref{Relative shear special flow} to \(\frac{\e}{50\h}\). By Lusin's theorem, there is a subset \(\tilde{W}\subset X^f\) with \(\mu^f(\tilde{W})>1-\frac{\e}{200 \h}\), such that \(\psi_1,...,\psi_{\h}\) are uniformly continuous on \(\tilde{W}.\) 
Recall \(\mathcal{Y}=X_{\frac{\e}{5\h}}\cap \tilde{W}\) in \eqref{Definition of Y}. 
By \Cref{Small shear general case}, suppose \(\nu\) is not 2-simple, then 
\begin{equation}
\label{n_k condistion 1 general case}
  \nu(G_k^c) \geq 1-\frac{\e}{500\h^2}  
\end{equation}
for large enough \(n_k.\)
Denote
\[\mathcal{Q}_{k}=\{(x,s)\in X^f: ((x,s),(y,s'))\in   G_k^c
 \textit{\;for every\;} (y,s')\in \Omega(x,s)\},\]
and \[\mathcal{U}_{k}=\{(x,s)\in X^f: x,y\in E_{n_k}\cap E_{n_k+1}\cap Z(\frac{\e}{500\h^2})\cap O(\frac{\e}{500\h^2}) \textit{\;for every\;} (y,s')\in \Omega(x,s)\}.\]
From the measure disintegration, 
\[\mu^f(\mathcal{Q}_{k})\geq 1-\frac{\e}{200 \h},\quad \mu^f(\mathcal{U}_{k})\geq 1-\frac{\e}{200\h}.\] 
Define \(\mathcal{Y}_k'=\mathcal{Y}\cap \mathcal{Q}_k \cap \mathcal{U}_k\). We have \(\mu^f(\mathcal{Y}_k')\geq 1-\frac{\e}{50\h}.\)
Recall \eqref{Type I distance} and \eqref{Type II distance}.
Pick \((x,s)\in \mathcal{Y}_k'\), by definition of \(\mathcal{Q}_k\), we can assume for every \(1\leq i \leq \h\), at least one of the following is true: 
\begin{equation}
\label{shorter scale orbit}
   \frac{1}{q_{n_k} \log q_{n_k}}<\min_{-q_{n_k}<j<q_{n_k}}\|x- R_{\alpha}^j y^i\|<\frac{5}{6 q_{n_k}},
\end{equation}
and
\begin{equation}
\label{larger scale orbit}
    \frac{1}{\r \log \r}<\min_{-q_{n_k}<j<q_{n_k}}\|x- R_{\alpha}^j y^i\|< \frac{5}{6 \r}.
\end{equation}
Recall \eqref{Coordinates of special flow}, for \((x,s)\in X^f\), \((y^i,s^i)\in \Omega(x,s)\) and \(t\in\mathbb{R}\), we let \(x_t=x+n(x,s,t)\alpha\) and \(y_t^i=y^i+n(y^i,s^i,t)\alpha\) for \(1\leq i \leq \h.\) If there exists \(1\leq i_* \leq \h\) such that \eqref{shorter scale orbit} holds, then for every \(1\leq\g,i \leq \h\), denote 
\[t_{\g}\coloneqq f^{(\g q_{n_k})}(x),\]
\[(x_{\g},s_{\g})\coloneqq(x_{t_{\g}},s_{t_{\g}}),\]
and
\[(y_{\g}^i,s_{\g}^i)\coloneqq(y_{t_{\g}}^i,s_{t_{\g}}^i).\] 
We also let \((x_0,s_0)\coloneqq (x,s)\) and \((y_0^i,s_0^i)\coloneqq(y^i,s^i)\). If \eqref{larger scale orbit} holds true for every \(1\leq i \leq \h\), then consider \[t_{\g}\coloneqq f^{(\g q_{n_{k+1}})}(x)\] instead.
By taking \(n_k\) large enough, suppose \((x_0,s_0)\) and \((x_1,s_1)\) are in \(\mathcal{Y}\), then for every \(1\leq i \leq \h\) there exists a number \(i'\) so that 
\begin{equation}
\label{n_k condition 2 general case}
   \|y_1^{i}-y_0^{i'}\|\leq \delta(\frac{\e}{50\h},N).  
\end{equation}
We say \((x,s)\) is \textit{k-comparable} if 
\[ \frac{2}{3}\|x_0-x_1\|<\|y_{1}^{i}-y_{0}^{i'}\|< \frac{4}{3}\|x_0-x_1\|\]
holds true for every \(1\leq i \leq \h.\)
Define \(S_{n}: X^f \rightarrow X^f\) to be a map such that
\[S_n(x,s)=(x+q_n\alpha,s).\]
By definition of \(\mathcal{U}_k\), \(x,y^{i}\in E_{n_k}\cap E_{n_k+1}\cap Z(\frac{\e}{500\h^2}) \cap O(\frac{\e}{500\h^2})\) for every \(1\leq i \leq \h\). We will pick \(n_k\) large enough so that \eqref{n_k condistion 1 general case} and \eqref{n_k condition 2 general case} hold simultaneously.
For such \(n_k,\) apply \Cref{Product Criteria} to \(\frac{\e}{50\h}\), the set \(W=\mathcal{Y}_k'\) and the map \(S=S_{n_k}\), \(S_{ n_{k}+1}\) and apply \Cref{Relative shear special flow} to \(((x_0,s_0),(y_0^{i'},s_0^{i'}))\) and \(((x_1,s_1),(y_1^{i},s_1^{i}))\), we deduce that the following set 
\begin{equation}
  \mathcal{Y}_k\coloneqq\{(x,s)\in \mathcal{Y}_k': (x,s) \textit{ is k-comparable}\}
\end{equation}
is of measure \(\mu^f(\mathcal{Y}_k)\geq 1- \frac{\e}{4\h}.\) Otherwise, we can find the set \(U=\mathcal{Y}_k^c\cap \mathcal{Y}_k'\) with \(\mu^f(U)\geq \frac{\e}{4\h}-\frac{\e}{50\h} \geq \frac{\e}{10\h}\) so that for \((x,s)\in U\), at least one of \eqref{Forward case induction} and \eqref{Backward case induction} is true. This contradicts the fact that \(\nu\) is not the product measure.
Now, pick \((x,s)\) in the following set
\[\mathcal{Y}_k \cap \left(\bigcap_{j=1}^{\h}S_{n_{k}}^{-j} \mathcal{Y}_k\right)\cap\left(\bigcap_{j=1}^{\h} S_{n_{k}+1}^{-j} \mathcal{Y}_k\right).\]
Suppose that there exists \(1\leq i_* \leq \h\) such that \eqref{shorter scale orbit} is true. Then, for any \(1\leq\g,i \leq \h\), there is a unique \(1\leq i'\leq \h\) such that
\begin{equation}
\label{Repeated k-comparable}
  \frac{2}{3\r}<\|y_{\g}^{i'}-y_{\g-1}^{i}\|< \frac{4}{3\r},  
\end{equation}
this is because \((x,s)\in \displaystyle \bigcap_{j=1}^{\h}S_{n_{k}}^{-j} \mathcal{Y}_k.\)
We claim that there are \(\g_1\) and \(\g_2\) such that
\begin{equation}
\label{Contradiction}
    \|y_{\g_1}^{i_*}-y_{\g_2}^{i_*}\| < \frac{4\h}{3 \r}.
\end{equation}
In fact, for \(\h+1\) points \(y_0^{i^*},y_1^{i^*},...,y_{\h}^{i*}\), they are all in neighborhoods of \(\h\) points \(y_0^1,y_0^2,...,y_0^{\h}.\) Hence, there are \(0\leq \g_1<\g_2 \leq \h\) such that \(y_{\g_1}^{i^*}\) and \(y_{\g_2}^{i^*}\) are in the same neighborhood. Applying \eqref{Repeated k-comparable} \(\g_2-\g_1\) times, then \eqref{Contradiction} follows.
However, on the other side, 
\[\frac{1}{2q_{n_k} \log q_{n_k}}<\min_{-2q_{n_k}<j<2q_{n_k}}\|x_{\g_1}- R_{\alpha}^j y_{\g_1}^{i_*}\|<\frac{5}{6 q_{n_k}}.\]
Using a similar proof in \Cref{Large Shearing}, there exist \(\mathfrak{d}_1^*,\mathfrak{d}_2^*>0\) such that
\[\mathfrak{d}_1^*\leq f^{((\g_2-\g_1)q_{n_k})}(x_{\g_1})-f^{((\g_2-\g_1)q_{n_k})}(y_{\g_1}^{i_*}) \leq \mathfrak{d}_2^*\h \log q_{n_k}.\]
Since the roof function is bounded below by \(a>0,\) there exists \(k^*\in \mathbb{N}\) such that
\[ q_{k^*-1} \leq \frac{\mathfrak{d}_2^* \h \log q_{n_k}}{a} \leq q_{k^*}. \]
Since \(q_{n+1} \leq q_{n}\log ^2q_{n}\) for every \(n,\) we have the following estimate:
\[\frac{1}{q_{k^*+1}} \geq \frac{1}{q_{{k}^*-1}\log ^4q_{{k}^*-1}} \geq \frac{a}{\mathfrak{d}_2^* \h \log^2 q_{n_k}}>\frac{100\h}{3 \r}.\]
Therefore,
\[\|y_{\g_1}^{i_*}-y_{\g_2}^{i_*}\| \geq \frac{1}{2q_{k^*+1}}> \frac{4\h}{3\r}.\]
This contradicts \eqref{Contradiction}. If \eqref{larger scale orbit} is true for every \(1\leq i \leq \h\), then set \(i_*=1.\) 
Then, for any \(1\leq\g,i \leq \h\), there is a unique \(1\leq i'\leq \h\) such that
\[\frac{1}{3q_{n_k+2}}<\|y_{\g}^{i'}-y_{\g-1}^{i}\|< \frac{4}{3q_{n_k+2}}.\]
Similarly, we can get
\[\|y_{\g_1}^{i_*}-y_{\g_2}^{i_*}\| < \frac{4\h}{3 q_{n_k+2}}\]
for some \(\g_1\) and \(\g_2.\)
The rest of the proof follows similarly. Hence, we complete the proof of \Cref{Main}.
\end{proof}

\bibliography{bib}
\bibliographystyle{alpha}
\printindex
\end{document}